\documentclass[10.5pt,a4paper]{article}
\let\originalss\ss

\usepackage{graphicx,latexsym,euscript,makeidx,color,bm}
\usepackage{amsmath,amsfonts,amssymb,amsthm,thmtools,mathtools,mathrsfs,enumerate}
\usepackage[colorlinks,linkcolor=blue,citecolor=red]{hyperref}

\usepackage{geometry}
\allowdisplaybreaks[4]
   \def\cA{{\cal A}}  
     
   \def\cC{{\cal C}}  
     
\def\dbE{\mathbb{E}}     
\def\dbF{\mathbb{F}}

   \def\cP{{\cal P}}
   
\def\dbR{\mathbb{R}}   
\def\dbS{\mathbb{S}}

   \def\cW{{\cal W}}

\def\ss{\smallskip}             \def\hb{\hbox}
\def\ms{\medskip}

\def\q{\quad}                \def\tr{\mathop{\hb{\rm tr}}\nolimits}
\def\qq{\qquad}             
\def\no{\noindent}

\def\rf{\eqref}

\def\3n{\negthinspace \negthinspace \negthinspace }
\def\2n{\negthinspace \negthinspace }
\def\1n{\negthinspace }
\def\bel{\begin{equation}\label}
\def\ee{\end{equation}}
\def\bea{\begin{eqnarray}}
\def\eea{\end{eqnarray}}
\def\bt{\begin{theorem}\label}
\def\et{\end{theorem}}
\def\bc{\begin{corollary}\label}
\def\ec{\end{corollary}}
\def\bex{\begin{example}\label}
\def\ex{\end{example}}
\def\bl{\begin{lemma}\label}
\def\el{\end{lemma}}
\def\bp{\begin{proposition}\label}
\def\ep{\end{proposition}}
\def\br{\begin{remark}\label}
\def\er{\end{remark}}
\def\ba{\begin{array}}
\def\ea{\end{array}}
\def\bde{\begin{definition}\label}
\def\ede{\end{definition}}

\def\d{\delta}

\def\k{\kappa}                          

\newtheoremstyle{thry}
{}
{}
{\sl}
{}
{\bf}
{.}
{.5em}
{}
\theoremstyle{thry}

\newtheorem{theorem}{Theorem}[section]

\newtheorem{proposition}[theorem]{Proposition}
\newtheorem{corollary}[theorem]{Corollary}
\newtheorem{lemma}[theorem]{Lemma}

\theoremstyle{definition}
\newtheorem{definition}[theorem]{Definition}
\newtheorem{example}[theorem]{Example}

\theoremstyle{remark}
\newtheorem{remark}[theorem]{Remark}

\def\punct{}
\newtheoremstyle{dotless}{}{}{\rm}{}{\bf}{\punct}{.5em}{}
\theoremstyle{dotless}
\newtheorem{taggedthmx}{}
\newenvironment{taggedthm}[1]
 {\def\punct{.}\renewcommand\thetaggedthmx{#1}\taggedthmx}
 {\endtaggedthmx}
\newtheorem{taggedassumptionx}{}
\newenvironment{taggedassumption}[1]
 {\renewcommand\thetaggedassumptionx{#1}\taggedassumptionx}
 {\endtaggedassumptionx}

\makeatletter

   \@addtoreset{equation}{section}

   \newcommand{\setword}[2]{%
   \phantomsection
   #1\def\@currentlabel{\unexpanded{#1}}\label{#2}%
   }
\makeatother

\def\Ent{\mathop{\hb{\rm Ent}}\nolimits}
\newcommand{\cNcal}{\mathcal N}

\begin{document}

\title{\bf Turnpike Properties for Exploratory Stochastic Linear-Quadratic Optimal Control with   Indefinite Cost Functionals}

\author{
Yaohua Hu\thanks{School of Mathematical Sciences, Shenzhen University, Shenzhen 518060, China, Email: mayhhu@szu.edu.cn.}
~~~
Zhaoqing Li\thanks{School of Mathematical Sciences, Shenzhen University, Shenzhen 518060, China, Email: lizhaoqing2024@mails.szu.edu.cn.}
~~~
Hanxiao Wang\thanks{School of Mathematical Sciences, Shenzhen University, Shenzhen 518060, China, Email: hxwang@szu.edu.cn. }
~~~
Jie Xiong\thanks{Department of Mathematics and SUSTech International center for Mathematics, Southern University of Science \& Technology, Shenzhen  518055, China, Email: xiongj@sustech.edu.cn.}
}

\date{\today}

\maketitle

\vskip-0.5cm

\no{\bf Abstract.}
Continuous-time exploratory stochastic control provides a stochastic control framework for the exploration--exploitation trade-off in reinforcement learning.
In this paper, we study the long-time behavior of finite-horizon exploratory stochastic linear-quadratic (LQ) optimal control problems as the time horizon \(T\to\infty\), and establish so-called \emph{turnpike properties} for such problems.
In the finite-horizon setting, we derive a feedback representation of the exploratory stochastic LQ problem and analyze its long-time structure.
We formulate an exploratory static distribution optimization problem analogous to the classical static optimization problem and identify the corresponding static reference object.
We prove a turnpike property for the optimal state process relative to a stochastic reference process, as well as a Wasserstein turnpike property for the optimal control distribution.
In addition to the exploratory control formulation, another key feature of our results is that, by introducing a  \emph{joint coercivity} condition, we allow the cost functional to be indefinite and the controlled state system to be merely stabilizable.
Consequently, the derived turnpike property is novel even within the classical stochastic LQ control framework.

\ms
\no{\bf Keywords.}
turnpike property, stochastic linear-quadratic control, exploratory control, entropy regularization, indefinite cost functional.

\ms
\no{\bf AMS 2020 Mathematics Subject Classification.} 49N10, 49N80, 93D23, 93E15, 93E20.

\ms

\section{Introduction}
\label{sec:introduction}

Let $(\Omega,\mathcal F,\mathbb P)$ be a complete probability space on which the Brownian motions used in this paper are defined.
Here $W=\{W(t):t\geq0\}$ is a standard one-dimensional Brownian motion, $\widehat W=\{\widehat W(t):t\geq0\}$ is a standard $n$-dimensional Brownian motion, and $\widetilde W=\{\widetilde W(t):t\geq0\}$ is a standard $m$-dimensional Brownian motion.
We assume that $W$ and $\widetilde W$ are independent.
Let $\mathbb F=\{\mathcal F_t\}_{t\geq0}$ be the usual augmentation of the natural filtration generated by these Brownian motions.
All random variables and processes appearing below are assumed to be measurable and adapted to $\mathbb F$;
admissible control distribution processes are  $\mathbb F$-progressively measurable.

This paper studies continuous-time exploratory stochastic linear-quadratic (LQ) control problems.
We start from a general controlled diffusion.
Consider the stochastic differential equation (SDE) with point-valued controls
\begin{equation*}
    \left\{
    \begin{aligned}
    &dX^u(t) = b(X^u(t),u(t))dt + \sigma(X^u(t),u(t))dW(t), \qquad t\geq0,\\
    &X^u(0)=x,
    \end{aligned}
    \right.
\end{equation*}
where $b$ and $\sigma$ are the drift and diffusion coefficients.
The usual measurability, growth, and integrability conditions are assumed to make this state equation well posed.
%If we take a control distribution $\pi\in\mathcal P(U)$ at state $x$, then the exploratory framework of %Wang--Zariphopoulou--Zhou \cite{WangZariphopoulouZhou2020} replaces the point-control coefficients in the state equation by %distributional averages
%\begin{equation*}
%    \widetilde b(x,\pi) = \int_U b(x,u)\,\pi(du),
%\end{equation*}
%and defines the exploratory diffusion through
%\begin{equation*}
%    \widetilde\sigma(x,\pi)\widetilde\sigma(x,\pi)^\top = \int_U \sigma(x,u)\sigma(x,u)^\top\,\pi(du).
%\end{equation*}
%Thus the general exploratory state equation can be written as
%\begin{equation}\label{explore-state}
%\left\{\begin{aligned}
 %   &dX^\pi(t) = \widetilde b(X^\pi(t),\pi_t)dt + \widetilde\sigma(X^\pi(t),\pi_t)d\widehat W(t), \\
%    & X^\pi(0)=x.
%    \end{aligned}\right.
%\end{equation}
%Let $\rho(x,u)$ be the running cost under a point control.
%Under a control distribution $\pi$, the exploratory running cost is defined by
%\begin{equation*}
 %   \widetilde \rho(x,\pi) = \int_U \rho(x,u)\,\pi(du).
%\end{equation*}
%If $\pi$ has density $p$, its differential entropy is
%$$
%\Ent(\pi)=-\int_U p(u)\log p(u)\,du.
%$$

%We now introduce the LQ framework. Let $U=\mathbb R^m$. The coefficients of the state equation \rf{explore-state} are given by
%\begin{equation}\label{LQ-coefficients}
 %   b(x,u)=Ax+Bu+b,\qq
%    \sigma(x,u)=Cx+Du+\sigma.
%\end{equation}
%The running cost is given by
%\begin{equation*}
%    \begin{aligned}
%    \rho(x,u) &=  x^\top Qx +2x^\top S^\top u +u^\top Ru +2q^\top x +2r^\top u .
%    \end{aligned}
%\end{equation*}

Following Wang--Zariphopoulou--Zhou
\cite{WangZariphopoulouZhou2020}, for a control distribution
$\pi\in\mathcal P(U)$ and a pointwise running cost $\rho(x,u)$,
we define the averaged drift, diffusion covariance, and running cost,
as well as the differential entropy when $\pi$ has density $p$:
\begin{equation*}
	\begin{alignedat}{2}
		&\widetilde b(x,\pi)=\int_U b(x,u)\,\pi(du),
		\qquad
		&&\widetilde\rho(x,\pi)=\int_U\rho(x,u)\,\pi(du),\\
		&\widetilde\sigma(x,\pi)\widetilde\sigma(x,\pi)^\top
		=\int_U\sigma(x,u)\sigma(x,u)^\top\,\pi(du),
		\qquad
		&&\Ent(\pi)=-\int_U p(u)\log p(u)\,du.
	\end{alignedat}
\end{equation*}
The exploratory state equation is
\begin{equation}\label{explore-state}
	\left\{
	\begin{aligned}
		&dX^\pi(t)=\widetilde b(X^\pi(t),\pi_t)dt
		+\widetilde\sigma(X^\pi(t),\pi_t)d\widehat W(t),\\
		&X^\pi(0)=x.
	\end{aligned}
	\right.
\end{equation}
We now specialize to the LQ setting with $U=\mathbb R^m$:
\begin{equation}\label{LQ-coefficients}
	\begin{aligned}
		&b(x,u)=Ax+Bu+b,\qquad
		\sigma(x,u)=Cx+Du+\sigma,\\
		&\rho(x,u)=x^\top Qx+2x^\top S^\top u+u^\top Ru
		+2q^\top x+2r^\top u.
	\end{aligned}
\end{equation}

In the exploratory framework, the control variable is  a process of probability distributions $\pi=\{\pi_t:t\in[0,T]\}$ on the control space $\mathbb R^m$.
Denote the mean and covariance of $\pi_t$ by
\begin{equation}\label{control-moments}
    \bar u(t)=\int_{\mathbb R^m}u\,\pi_t(du),\qq
    \Lambda(t)=\int_{\mathbb R^m}(u-\bar u(t))(u-\bar u(t))^\top\pi_t(du).
\end{equation}
Substituting the LQ coefficients \eqref{LQ-coefficients} and the control moments \eqref{control-moments} into the general exploratory state equation \eqref{explore-state} gives the following equivalent exploratory LQ state equation:
\begin{equation*}
    \left\{
    \begin{aligned}
    &dX^\pi(t) = \bigl(AX^\pi(t)+B\bar u(t)+b\bigr)dt+ \bigl(CX^\pi(t)+D\bar u(t)+\sigma\bigr)dW(t)\\
    &\qq\qq~ + D\Lambda(t)^{1/2}d\widetilde W(t), \qquad t\in[0,T],\\
    &X^\pi(0)=x.
    \end{aligned}
    \right.
\end{equation*}
Here $W$ represents the original system noise and $\widetilde W$ represents the additional noise generated by the covariance of the randomized control.
%The entropy-regularized exploratory LQ cost functional on the finite horizon $[0,T]$ is given by
%\begin{equation*}
 %   \begin{aligned}
%    &J_T^\lambda(x;\pi) = \mathbb E\int_0^T \int_{\mathbb R^m} \Bigl[ X^\pi(t)^\top QX^\pi(t) +2X^\pi(t)^\top S^\top u +u^\top Ru+2q^\top X^\pi(t) \\
 %   &\qquad\qquad\q+2r^\top u \Bigr]\pi_t(du)\,dt  + \lambda\mathbb E\int_0^T\int_{\mathbb R^m}p_t(u)\log p_t(u)\,du\,dt ,
%    \end{aligned}
%\end{equation*}
%where $p_t$ is the density of $\pi_t$ and $\lambda>0$ is the exploration weight.
%The entropy regularization term $\lambda\int_{\mathbb R^m}p_t(u)\log p_t(u)\,du$ is precisely $-\lambda\Ent(\pi_t)$.
%Since the cost is minimized, this term favors larger entropy and promotes randomized exploration.
%Thus $\lambda$ controls the strength of exploration: a larger $\lambda$ gives more weight to exploration, while a smaller %$\lambda$ brings the model closer to the classical point-control problem.

The entropy-regularized cost on $[0,T]$ is
\begin{equation*}
	\begin{aligned}
		&J_T^\lambda(x;\pi)
		=\mathbb E\int_0^T\int_{\mathbb R^m}
		\Bigl[
		X^\pi(t)^\top QX^\pi(t)
		+2X^\pi(t)^\top S^\top u+u^\top Ru\\
		&\q\qquad\qquad
		+2q^\top X^\pi(t)+2r^\top u
		\Bigr]\,\pi_t(du)\,dt
		-\lambda\mathbb E\int_0^T\Ent(\pi_t)\,dt,
	\end{aligned}
\end{equation*}
where $\lambda>0$ is the exploration weight.
Since the cost is minimized, the entropy term encourages exploration:
a larger $\lambda$ gives more weight to exploration, while a smaller
$\lambda$ brings the model closer to the classical point-control problem.

The finite-horizon exploratory LQ problem can be stated as follows.

\begin{taggedthm}{Problem $(\mathrm{ESLQ})_T$}\label{prob:exploratory-finite}
For a given initial state $x\in\mathbb R^n$, find $\bar\pi^T$ in the admissible class of exploratory control distributions such that
\begin{equation*}
    J_T^\lambda(x;\bar\pi^T) = \inf_{\pi}J_T^\lambda(x;\pi) \equiv V_T^\lambda(x).
\end{equation*}
\end{taggedthm}
If such a distribution process exists, it is called an optimal exploratory control distribution of \ref{prob:exploratory-finite}, and the corresponding state process is denoted by $\bar X^T$.

Our main question concerns the long-time behavior of the optimal pair $(\bar X^T,\bar\pi^T)$ as $T\to\infty$.
More precisely, we seek static or stationary reference objects that are independent of $T$, such that the finite-horizon optimal processes remain close to these objects for most of the interval $[0,T]$.
Since the control object in this paper is a probability distribution, a reasonable turnpike property should compare not only the state and the mean control but also the full control distributions.
Thus, we use the Wasserstein distance $\cW_2$ to measure the difference between control distributions.

The main result of this paper is that, under the stabilizable condition \ref{ass:L2-stabilizable} and the joint coercivity condition \ref{ass:JC},
 we establish an estimate of the form
\begin{equation}\label{T-P1}
    \mathbb E|\bar X^T(t)-X^*(t)|^2 + \mathbb E \cW_2^2 \bigl( \bar\pi_t^T,\pi_t^* \bigr) \leq K\left(e^{-\mu t}+e^{-\mu(T-t)}\right),
    \qquad t\in[0,T],
\end{equation}
%where $K,\mu>0$ are independent of $T,t$, and $X^*$ and $\pi^*$ are suitable long-time reference processes and reference %control distributions, which can be uniquely determined by solving a static optimization problem
%(see Theorem \ref{Thm:TP}).
where $K,\mu>0$ are independent of $T,t$, and the reference
	process $X^*$ and control distributions $\pi^*$ are constructed from the static
	optimizer and the limiting feedback parameters
	(see Theorem~\ref{Thm:TP}).
This estimate is called a \emph{turnpike property} for the exploratory stochastic LQ problem.
It means that, away from the initial and terminal endpoints, the finite-horizon optimal objects are exponentially close to the  reference objects (see Corollary \ref{coroll:TP}).

The study of turnpike properties began with long-run economic growth problems \cite{Dorfman-Samuelson-Solow1958,McKenzie1976,Neumann1945}, and has since been developed widely in optimal control, variational problems, and control of partial differential equations; see \cite{Faulwasser-Grune2022,LewisPaul2019,Porretta-Zuazua2013,Sethi-Soner-Zhang-Jiang1992,Trelat-Zhang2018,TrelatZuazua2025Survey,Trelat-Zuazua2015,Zaslavski2006}.
For deterministic and discrete-time LQ problems, the static optimization problem, the Riccati equation, and the closed-loop feedback have a clear relation.
Hence the turnpike property can be described by a static optimal point; related results can be found in \cite{Breiten-Pfeiffer2020,Damm-Grune-Stieler-Worthmann2014,Grune-Guglielmi2018,Grune-Guglielmi2021}.
For  stochastic LQ problems,
Sun--Wang--Yong \cite{SunWangYong2022} studied turnpike properties for stochastic LQ problems and pointed out that the correct form of the static problem is a key difference between the stochastic and deterministic cases.
Furthermore, Sun--Yong \cite{SunYong2022} proved turnpike properties for mean-field LQ problems and stressed that the proper reference object in stochastic systems is usually not a deterministic point, but a $T$-independent stochastic reference process.
Further studies of stochastic LQ turnpike properties can be found in \cite{Bayraktar-Jian2026,HuWu2026,MeiWangYong2025,SchiesslOuFaulwasserBaumannGrune2025,SunYong2024ZeroSum}.

Alongside these developments in turnpike theory, the entropy-regularized exploratory-control framework of Wang--Zariphopoulou--Zhou \cite{WangZariphopoulouZhou2020} has led to a growing literature on  reinforcement learning. In particular,
Jia--Zhou \cite{JiaZhou2022PE,JiaZhou2022PG,JiaZhou2023qLearning} systematically developed a martingale approach to policy evaluation and temporal-difference learning, and subsequently established policy-gradient, actor--critic, and continuous-time $q$-learning theories under the exploratory formulation.
In LQ settings, Hambly--Xu--Yang \cite{HamblyXuYang2021} studied policy-gradient methods for finite-horizon noisy problems in discrete time.
For entropy-regularized LQ problems, Guo--Li--Xu \cite{GuoLiXu2026} and Giegrich--Reisinger--Zhang \cite{GiegrichReisingerZhang2024} analyzed policy learning in discrete and continuous time, respectively.
Huang--Jia--Zhou \cite{HuangJiaZhou2025} obtained a sublinear regret bound for a model-free algorithm, and Szpruch--Treetanthiploet--Zhang \cite{SzpruchTreetanthiploetZhang2024} studied the optimal scheduling of the entropy regularizer.
Huang--Zhou \cite{HuangZhou2025DataDriven} further proposed a data-driven adaptive exploration mechanism.
% The exploratory LQ framework also has applications to mean--variance portfolio selection;
% see Wang--Zhou \cite{WangZhou2020MV}, Dai--Dong--Jia \cite{DaiDongJia2023},  and Huang--Jia--Zhou \cite{HuangJiaZhou2026MV}, for example.
The literature also includes studies of mean-field games
	\cite{GuoHuXuZhang2023,GuoXuZariphopoulou2022}
	and applications to mean--variance portfolio selection
	\cite{DaiDongJia2023,HuangJiaZhou2026MV,WangZhou2020MV}.
These works mainly study learning and exploration mechanisms.
This paper focuses on the turnpike property, namely the estimate \rf{T-P1}, for exploratory stochastic LQ control. This can be viewed as a fine-grained behavioral characterization of the optimal state processes and optimal control distributions.

Compared with existing results, our work has two main features.

(i) In the exploratory framework, the control variable is a process of probability distributions.
Thus, the turnpike property of the controls in our paper needs to be established in the Wasserstein space.

To this end, we propose an exploratory static problem defined over general distributions rather than over the Euclidean space as in \cite{SunWangYong2022}.
Specifically, it takes the following form:
\begin{equation}
    \label{distribution-problem-intro}
    \begin{aligned}
    \min_{x,\pi}\quad & \int_{\dbR^m}F_0(x,u)\,\pi(du)-\lambda\Ent(\pi)\\
    \text{s.t.}\quad & Ax+B\bar u_\pi+b=0,
    \end{aligned}
\end{equation}
where
\begin{equation}\label{F0}
    F_0(x,u)=x^\top Qx+2x^\top S^\top u+u^\top Ru+2q^\top x+2r^\top u+(Cx+Du+\sigma)^\top P_\infty(Cx+Du+\sigma),
\end{equation}
%with $P_\infty$ being the unique static stabilizing solution of the corresponding algebraic Riccati equation, and
with $P_\infty$ being the unique stabilizing solution of the
algebraic Riccati equation \eqref{ARE}, and
\begin{equation*}
    \begin{aligned}
    &\bar u_\pi=\int_{\mathbb R^m}u\,\pi(du),\q
    \Lambda_\pi=\int_{\mathbb R^m}(u-\bar u_\pi)(u-\bar u_\pi)^\top\pi(du),\q
    \Ent(\pi)=-\int_{\mathbb R^m}p(u)\log p(u)\,du,
    \end{aligned}
\end{equation*}
with $\pi(du)=p(u)\,du$.

Note that in \rf{distribution-problem-intro}, $\pi$ ranges over proper probability distributions,
and the equality constraint is expressed in terms of the mean of $\pi$.
Thus, it is still an infinite-dimensional constrained optimization problem.
Using \cite[Theorem 8.6.5]{CoverThomas2006} (see Lemma~\ref{lem:maximum-entropy}),
we show that the exploratory static distribution optimization problem is equivalent to the following
finite-dimensional one:
\begin{equation}
\label{Gaussian-problem-Intro}
    \begin{aligned}
    \min_{x,u,\Lambda}\quad & F_0(x,u) + \tr(M_\infty\Lambda) - \frac{\lambda}{2}\log\big((2\pi e)^m\det\Lambda\big)\\
    \text{s.t.}\quad & Ax+Bu+b=0,
    \end{aligned}
\end{equation}
where $M_\infty:=R+D^\top P_\infty D$, and $u$ and $\Lambda$ are the mean and covariance of the distribution $\pi$, respectively.

(ii) The second feature is that the cost functional is indefinite, meaning that the weighting matrices need not be positive semidefinite.
%, the matrices $Q$ and $P_\infty$ in \rf{F0} may be indefinite,
%which causes essential difficulties in solving the static optimization problem \rf{Gaussian-problem-Intro}.
We remark that the indefinite case is also novel even for the turnpike property of classical stochastic LQ control problems without exploration.

We assume that the controlled system $[A,C;B,D]$ is stabilizable (see \ref{ass:L2-stabilizable}),
which includes the stable system as a special case.
A main contribution of this paper is the introduction of the so-called \emph{joint coercivity} condition (see \ref{ass:JC}).
%For the finite-horizon problem and for the infinite-horizon problem with a stable controlled system,
%this condition is equivalent to the well-known \emph{uniform convexity} condition (see Remark \ref{remark2.10} and %Proposition \ref{prop:JC-UC}) introduced by Sun--Li--Yong \cite{Sun-Li-Yong2016}.
On each fixed finite horizon, this condition is equivalent to the well-known \emph{uniform convexity} condition \cite{Sun-Li-Yong2016}, with constants that may depend on the	horizon (see Remark~\ref{remark2.10}).
If the uncontrolled system is $L^2$-stable, joint coercivity is	equivalent to uniform convexity with a constant independent of the horizon (see Proposition~\ref{prop:JC-UC}).
Thus, the joint coercivity condition is a natural extension of the uniform convexity condition,
enabling us to handle infinite-horizon stabilizable systems in the indefinite cost functional case.

The indefinite cost functional gives rise to several significant difficulties in establishing the turnpike property.
Since the matrices $Q$ and $P_\infty$ in \rf{F0} may be indefinite,
solving the static optimization problem \rf{Gaussian-problem-Intro} is itself nontrivial.
By a careful convexity analysis of $F_0$, we show that problem \rf{Gaussian-problem-Intro} still admits a unique solution $(x_*,u_*, \Lambda_*)$, even when $P_\infty$ is indefinite (see Propositions \ref{prop:static-KKT} and \ref{prop:optimal-covariance}).
Then, under Assumptions \ref{ass:L2-stabilizable}--\ref{ass:JC}, we prove that \eqref{distribution-problem-intro} admits a unique optimal solution $(x_*,\pi_*^\lambda)$, with the optimal control distribution being Gaussian, without assuming Gaussianity a priori:
\begin{equation*}
    \pi_*^\lambda = \mathcal N \left( u_*, \frac{\lambda}{2} (R+D^\top P_\infty D)^{-1} \right).
\end{equation*}
This shows that the exploratory regularization does not shift the static center.
Instead, it yields Gaussian exploration around the static control point $u_*$, with covariance determined by $\lambda$ and  $P_\infty$.

To establish the turnpike property \rf{T-P1} (see Theorem \ref{Thm:TP}), we use the unique optimal solution $(x_*,\pi_*^\lambda)$ to problem \eqref{distribution-problem-intro}, establish the convergence of the Riccati equations (see Lemma \ref{lem:riccati-convergence} and Proposition \ref{prop:riccati-convergence}), and exploit certain refined properties of the Wasserstein distance between Gaussian distributions (see Lemma \ref{lem:Gaussian-W2}).
It is worth pointing out that, since the cost functional is indefinite, the method employed by Sun--Wang--Yong \cite{SunWangYong2022} cannot be applied to establish the convergence of the Riccati equations in Lemma \ref{lem:riccati-convergence} and Proposition \ref{prop:riccati-convergence}.
We transform the stabilizable system into a stable one and, inspired by Sun--Yong \cite{SunYong2024ZeroSum}, overcome this difficulty by means of a Hilbert-space approach.
More importantly, we prove that the convergence of the Riccati equations is equivalent to the satisfaction of the joint coercivity condition (see Theorem \ref{Thm-equiv}).
Since the convergence of the Riccati equations is fundamental to the turnpike property, this shows that the joint coercivity condition introduced in this paper is sharp.
%Under $L^2$-stabilizability and within the class of strongly regular
%	finite-horizon Riccati solutions, Proposition~\ref{prop:riccati-convergence}
%	and Theorem~\ref{Thm-equiv} characterize joint coercivity by convergence
%	to a stabilizing ARE solution with $R+D^\top P_\infty D>0$.
%	Thus, joint coercivity is sharp for this stabilizing Riccati limit.

The remainder of the paper is organized as follows.
Subsection~\ref{sec:preliminaries-exploratory-slq} presents the precise formulation of the exploratory stochastic LQ problem.
Subsection~\ref{sec:infinite-horizon-problem} reviews the stability of stochastic linear systems and derives the optimal feedback for the infinite-horizon stochastic LQ problem.
Section~\ref{sec:static-distribution-optimization} establishes the existence and uniqueness of the solution to the classical static problem without requiring the cost functional to be positive semidefinite, and constructs its unique Gaussian exploratory extension.
Section~\ref{sec:riccati-parameter-convergence} proves exponential convergence of the solutions of the finite-horizon Riccati equations, the feedback parameters, and the exploration covariance to their infinite-horizon counterparts.
Section~\ref{sec:turnpike-properties} establishes the turnpike property for the optimal state process and the Wasserstein turnpike property for the optimal control distributions.

\section{Preliminaries}
\subsection{Exploratory Stochastic LQ Problems}
\label{sec:preliminaries-exploratory-slq}

We first fix some notation used below.
Let $\dbR^{n\times m}$ be the space of $n\times m$ real matrices, with the Frobenius inner product $\langle M,N\rangle=\tr(M^\top N)$, $M,N\in\dbR^{n\times m}$.
The norm induced by this inner product is denoted by $|M|$ or $\|\cdot\|_F$.
For matrices, $\|\cdot\|$ denotes the spectral norm; for linear operators,
	$\|\cdot\|$ denotes the induced operator norm.
Let $\dbS^k$ be the space of $k\times k$ real symmetric matrices, and let $\dbS^k_+$ and $\dbS^k_{++}$ be the sets of positive semidefinite and positive definite matrices, respectively.
For $P,Q\in\dbS^k$, we write $P\ge Q$ if $P-Q\in\dbS^k_+$, and $P>Q$ if $P-Q\in\dbS^k_{++}$.
The identity matrix is denoted by $I_k$.
We use the probability space and filtration introduced in Section~\ref{sec:introduction}, and all processes involved are assumed to be $\dbF$-progressively measurable.

%The control object in this paper is a probability distribution on the control space $\dbR^m$.
%Let
%\[
%    \cP_2(\dbR^m) = \left\{ \pi\in\cP(\dbR^m)\ \middle|\ \int_{\dbR^m}|u|^2\,\pi(du)<\infty \right\}.
%\]
%For $\pi\in\cP_2(\dbR^m)$, define
%\begin{equation*}
%    \bar u(\pi)=\int_{\dbR^m}u\,\pi(du),
%\end{equation*}
%and
%\begin{equation*}
%    \Lambda(\pi) = \int_{\dbR^m} (u-\bar u(\pi))(u-\bar u(\pi))^\top\,\pi(du).
%\end{equation*}
%Then
%\begin{equation*}
%    \int_{\dbR^m}uu^\top\,\pi(du) = \bar u(\pi)\bar u(\pi)^\top+\Lambda(\pi).
%\end{equation*}
%If $\pi$ is absolutely continuous with respect to the Lebesgue measure, with density $p$, then its differential entropy is %defined by
%\begin{equation*}
%    \Ent(\pi) = -\int_{\dbR^m}p(u)\log p(u)\,du.
%\end{equation*}
%For a Gaussian distribution $\pi=\cNcal(u,\Lambda)$, where $\Lambda\in\dbS^m_{++}$, we have
%\begin{equation*}
%    \Ent\big(\cNcal(u,\Lambda)\big) = \frac12\log\big((2\pi e)^m\det\Lambda\big).
%\end{equation*}

%Now we state the finite-horizon exploratory stochastic LQ problem which is studied in this paper.
%Let $A,C\in\dbR^{n\times n}$, $B,D\in\dbR^{n\times m}$, $b,\sigma,q\in\dbR^n$, $r\in\dbR^m$, $Q\in\dbS^n$, $R\in\dbS^m$, and %$S\in\dbR^{m\times n}$.
%Given a control distribution process $\pi=\{\pi_t\}_{0\le t\le T}$, set $\bar u(t)=\bar u(\pi_t)$ and %$\Lambda(t)=\Lambda(\pi_t)$.
%The exploratory state process $X^\pi$ satisfies

Let $\cP_2(\dbR^m)$ denote the probability distributions on
$\dbR^m$ with finite second moment. For $\pi\in\cP_2(\dbR^m)$,
define its mean and covariance by
\begin{equation*}
	\bar u(\pi)=\int_{\dbR^m}u\,\pi(du),
	\qquad
	\Lambda(\pi)=\int_{\dbR^m}
	(u-\bar u(\pi))(u-\bar u(\pi))^\top\,\pi(du).
\end{equation*}
Then $\int_{\dbR^m}uu^\top\,\pi(du)
=\bar u(\pi)\bar u(\pi)^\top+\Lambda(\pi)$.
If $\pi$ has a Lebesgue density $p$, its differential entropy
and the Gaussian entropy formula are
\begin{equation*}
	\Ent(\pi)=-\int_{\dbR^m}p(u)\log p(u)\,du,
	\qquad
	\Ent\big(\cNcal(u,\Lambda)\big)
	=\frac12\log\big((2\pi e)^m\det\Lambda\big),
\end{equation*}
where $\Lambda\in\dbS^m_{++}$ in the Gaussian case.

For the finite-horizon LQ problem, let
$A,C\in\dbR^{n\times n}$, $B,D\in\dbR^{n\times m}$,
$b,\sigma,q\in\dbR^n$, $r\in\dbR^m$,
$Q\in\dbS^n$, $R\in\dbS^m$, and $S\in\dbR^{m\times n}$.
Given a control distribution process
$\pi=\{\pi_t\}_{0\le t\le T}$, set
$\bar u(t)=\bar u(\pi_t)$ and $\Lambda(t)=\Lambda(\pi_t)$.
The exploratory state process $X^\pi$ satisfies

\begin{equation}\label{state}
    \left\{
    \begin{aligned}
    &dX^\pi(t) = \bigl(AX^\pi(t)+B\bar u(t)+b\bigr)dt+ \bigl(CX^\pi(t)+D\bar u(t)+\sigma\bigr)dW(t)\\
    &\qq\qq\q + D\Lambda(t)^{1/2}d\widetilde W(t), \qquad t\in[0,T],\\
    &X^\pi(0)=x.
    \end{aligned}
    \right.
\end{equation}
For $x\in\dbR^n$ and $\pi\in\cP_2(\dbR^m)$, define
the running cost under the control distribution $\pi$ by
\begin{equation*}
	\ell(x,\pi)
	=\int_{\dbR^m}
	\bigl[
	x^\top Qx+2x^\top S^\top u+u^\top Ru
	+2q^\top x+2r^\top u
	\bigr]\,\pi(du).
\end{equation*}

For a given initial state $x\in\dbR^n$ and exploration weight $\lambda>0$, a control distribution process $\pi=\{\pi_t\}_{0\le t\le T}$ belongs to $\cA_T^\lambda(x)$ if it satisfies the following conditions:
\begin{enumerate}[(i)]
    \item for each $t\in[0,T]$, $\pi_t\in\cP_2(\dbR^m)$ and $\pi_t$ is absolutely continuous with respect to the Lebesgue measure, with density $p_t$;
    \item for every Borel set $E\subset\dbR^m$, the process $t\mapsto \pi_t(E)$ is $\dbF$-progressively measurable;
    \item the mean and covariance satisfy $\dbE\int_0^T\big(|\bar u(t)|^2+\tr\Lambda(t)\big)dt<\infty$;
    \item the entropy regularization term is integrable, namely $\dbE\int_0^T\left|\int_{\dbR^m}p_t(u)\log p_t(u)\,du\right|dt<\infty$;
    \item under $\pi$, the state equation \eqref{state} is well posed.
\end{enumerate}
For the linear state equation, condition (iii) yields the usual second-moment estimate for $X^\pi$.
Hence, together with condition (iv), the above assumptions ensure that all terms in the cost functional below are integrable.

\begin{taggedthm}{\ref{prob:exploratory-finite}}
For a given $x\in\dbR^n$, over all $\pi\in\cA_T^\lambda(x)$, minimize
\begin{equation*}
    J_T^\lambda(x;\pi) = \dbE\int_0^T \left[ \ell(X^\pi(t),\pi_t) + \lambda\int_{\dbR^m}p_t(u)\log p_t(u)\,du \right]dt.
\end{equation*}
The value function is defined by
\begin{equation*}
    V_T^\lambda(x) = \inf_{\pi\in\cA_T^\lambda(x)}J_T^\lambda(x;\pi).
\end{equation*}
\end{taggedthm}

\begin{remark}
If there exists $\pi^{T,\lambda,*}\in\cA_T^\lambda(x)$ such that $J_T^\lambda(x;\pi^{T,\lambda,*})=V_T^\lambda(x)$, then $\pi^{T,\lambda,*}$ is called an \emph{optimal control distribution} of \ref{prob:exploratory-finite}.
\end{remark}

The exploratory control framework and the entropy-regularized objective were first introduced  in \cite{WangZariphopoulouZhou2020}.
The exploratory HJB equation and its well-posedness theory can be found in \cite{TangZhangZhou2021,WangZariphopoulouZhou2020}.
In this section, we use these results only in the LQ case.
%Let $V=V_T^\lambda(t,x)$ be the value function from time $t$ and state $x$ to terminal time $T$.
%If $V$ is smooth enough, define
%\begin{equation}
%    \begin{aligned}
%    \Psi_V(t,x,u) =&\ell_0(x,u) +\langle V_x(t,x),Ax+Bu+b\rangle\\
%    &+\frac12\tr\!\left[ (Cx+Du+\sigma)(Cx+Du+\sigma)^\top V_{xx}(t,x) \right].
%    \end{aligned}
%\end{equation}
% The dynamic programming principle gives the following exploratory HJB equation, of the same form as in \cite{TangZhangZhou2021}:
%\begin{equation}
%    \begin{aligned}
%    0&=\partial_t V(t,x) +\inf_{\substack{p\ge0,\ \int_{\dbR^m}p(u)du=1\\
%    \int_{\dbR^m}|u|^2p(u)du<\infty,\ \int_{\dbR^m}|p(u)\log p(u)|du<\infty}} \int_{\dbR^m} \left[ %\Psi_V(t,x,u)+\lambda\log p(u) \right]p(u)du.
%    \end{aligned}
%\end{equation}
%Thus, after fixing $(t,x)$, the optimal feedback density is obtained by minimizing under the constraints $\int_{\dbR^m}p(u)du=1$ and $p(u)\ge0$, namely
%\begin{equation}
%    \label{Gibbs-feedback}
%    p_T^{\lambda,*}(u\mid t,x) = \frac{ \exp\!\left[-\lambda^{-1}\Psi_{V_T^\lambda}(t,x,u)\right] }{ \int_{\dbR^m}\exp\!\left[-\lambda^{-1}\Psi_{V_T^\lambda}(t,x,v)\right]dv }.
%\end{equation}
%When the normalizing constant is finite, \eqref{Gibbs-feedback} is the unique minimizing density of the %pointwise problem.
Consider the following Riccati equation:
\begin{equation}
    \label{Riccati-T}
    \left\{
    \begin{aligned}
    &\dot P_T(t)+P_T(t)A+A^\top P_T(t)+C^\top P_T(t)C+Q-\big(P_T(t)B+C^\top P_T(t)D+S^\top\big)\\
    &\quad \cdot\big(R+D^\top P_T(t)D\big)^{-1}\big(B^\top P_T(t)+D^\top P_T(t)C+S\big)=0, \qquad t\in[0,T],\\
    &P_T(T)=0.
    \end{aligned}
    \right.
\end{equation}
A solution $P_T\in C([0,T];\dbS^n)$ of \eqref{Riccati-T} is called \emph{strongly regular} if there exists a constant $\delta_T>0$ such that
\begin{equation*}
    R+D^\top P_T(t)D\ge\delta_T I_m, \qquad t\in[0,T].
\end{equation*}
The next proposition gives the Gaussian feedback law and the
linear terminal-value equation for $\varphi_T$.
Related results and proof details can be found in \cite{GiegrichReisingerZhang2024,Sun-Li-Yong2016,TangZhangZhou2021,WangZariphopoulouZhou2020}.

\begin{proposition}
    \label{prop:exploratory-solution}
    Let $P_T$ be a strongly regular solution of the Riccati equation \eqref{Riccati-T}.
    Define
    \begin{equation*}
    \begin{aligned}
    &M_T(t)=R+D^\top P_T(t)D,\q
    N_T(t)=B^\top P_T(t)+D^\top P_T(t)C+S,\q
    K_T(t)=M_T(t)^{-1}N_T(t).
    \end{aligned}
\end{equation*}
    Then the following linear terminal-value problem admits a unique solution $\varphi_T$:
    \begin{equation}\label{linear-terminal-equation}
    \left\{
    \begin{aligned}
    &\dot\varphi_T(t) +\big(A-BK_T(t)\big)^\top\varphi_T(t) +\big(C-DK_T(t)\big)^\top P_T(t)\sigma\\
    &\quad-K_T(t)^\top r+P_T(t)b+q=0, \qquad t\in[0,T],\\
    &\varphi_T(T)=0.
    \end{aligned}
    \right.
\end{equation}
    Set
    \begin{equation*}
    \begin{aligned}
    d_T(t)=B^\top\varphi_T(t)+D^\top P_T(t)\sigma+r,\qq
    \ell_T(t)=M_T(t)^{-1}d_T(t),\qq
    \Lambda_T^\lambda(t)=\frac{\lambda}{2}M_T(t)^{-1}.
    \end{aligned}
\end{equation*}
    Then
    %the value function of \ref{prob:exploratory-finite} has the quadratic form
    %\begin{equation}
   % V_T^\lambda(x) = x^\top P_T(0)x+2\varphi_T(0)^\top x+\eta_T^\lambda(0),
%\end{equation}
 %   where $\eta_T^\lambda$ is uniquely determined by the remaining constant-order equation and $\eta_T^\lambda(T)=0$.
    the optimal feedback control distribution is given by
    \begin{equation}
    \label{optimal-feedback}
        \pi_T^{\lambda,*}(du\mid t,x)
    =\cNcal\big(-K_T(t)x-\ell_T(t),\Lambda_T^\lambda(t)\big)(du).
    \end{equation}
\end{proposition}

\begin{remark}
   % By the Gaussian feedback law \eqref{optimal-feedback}, the  mean of optimal feedback $-K_T(t)x-\ell_T(t)$ is %independent of $\lambda$, while the exploration covariance $\Lambda_T^\lambda(t)=\frac{\lambda}{2}M_T(t)^{-1}$ is %proportional to $\lambda$.
   By \eqref{optimal-feedback}, the mean $-K_T(t)x-\ell_T(t)$ of the optimal feedback distribution is independent of $\lambda$, while its covariance $\Lambda_T^\lambda(t)=\frac{\lambda}{2}M_T(t)^{-1}$ is proportional to $\lambda$.
\end{remark}

\subsection{Infinite-Horizon Problem}
\label{sec:infinite-horizon-problem}

We  recall the stability of stochastic linear systems.
Consider the homogeneous uncontrolled system
\begin{equation*}\left\{\begin{aligned}
    &dX(t)=AX(t)dt+CX(t)dW(t),\qquad t\ge0,\\
    &X(0)=x.
    \end{aligned}\right.
\end{equation*}
For any initial value $x\in\dbR^n$, denote its solution by $X(\cdot;x)$.

\begin{definition}
(i) The system $[A,C]$ is called $L^2$-stable if, for every $x\in\dbR^n$,
\begin{equation*}
    \dbE\int_0^\infty |X(t;x)|^2dt<\infty.
\end{equation*}
 (ii) The system $[A,C]$ is called mean-square exponentially stable if there exists $\beta>0$ such that, for every $x\in\dbR^n$,
\[
    \sup_{t\ge0} e^{\beta t}\dbE|X(t;x)|^2<\infty.
\]
\end{definition}

%Let $\Phi(\cdot)$ be the solution of the matrix-valued stochastic differential equation
%\begin{equation}\label{fundamental-matrix}
 %   \left\{
%    \begin{aligned}
 %   &d\Phi(t)=A\Phi(t)dt+C\Phi(t)dW(t),\qquad t\ge0,\\
%    &\Phi(0)=I_n .
%    \end{aligned}
%    \right.
%\end{equation}
%The $L^2$-stability of stochastic linear systems can be characterized by the following equivalent conditions; see %\cite{SunYong2018}.

%\begin{lemma}
%\label{lem:stability-equivalence}
%Let $\Phi(t)$ solve \eqref{fundamental-matrix}. The following conditions are equivalent:
%\begin{enumerate}[(i)]
%    \item $[A,C]$ is mean-square exponentially stable;
%    \item there exist constants $\alpha,\beta>0$ such that $\dbE|\Phi(t)|^2\le \alpha e^{-\beta t}$, $t\ge0$;
%    \item $\displaystyle \int_0^\infty\dbE|\Phi(t)|^2dt<\infty$;
%    \item $[A,C]$ is $L^2$-stable;
 %   \item there exists $P\in\dbS^n_{++}$ such that $PA+A^\top P+C^\top PC<0$.
%\end{enumerate}
%\end{lemma}
The following characterization will be used below;
	see Sun--Yong \cite{SunYong2018}.
	\begin{lemma}\label{lem:stability-equivalence}
		The following conditions are equivalent:
		\begin{enumerate}[(i)]
			\item $[A,C]$ is mean-square exponentially stable;
			\item $[A,C]$ is $L^2$-stable;
			\item there exists $P\in\dbS^n_{++}$ such that
			$PA+A^\top P+C^\top PC<0$.
		\end{enumerate}
\end{lemma}

Moreover, from Sun--Wang--Yong \cite[Corollary~3.1]{SunWangYong2022}, we have the following result.

\begin{corollary}
\label{cor:Hurwitz}
If $[A,C]$ is $L^2$-stable, then there exist constants $\alpha,\beta>0$ such that $|e^{At}|\le \alpha e^{-\beta t/2}$, $t\ge0$.
\end{corollary}

Now we consider the following controlled homogeneous linear system:
\begin{equation}\label{state-homo}
	\left\{
	\begin{aligned}
		&dX(t)=\bigl(AX(t)+Bu(t)\bigr)dt +\bigl(CX(t)+Du(t)\bigr)dW(t),\qquad t\in[0,\infty),\\
		&X(0)=x.
	\end{aligned}
	\right.
\end{equation}
For the infinite-horizon problem, the admissible control set is
\begin{equation}\label{admissible}
	\mathcal U_{\mathrm{ad}}(x)
	=\left\{u\in L^2_{\dbF}(0,\infty;\dbR^m):
	X^{x,u}\in L^2_{\dbF}(0,\infty;\dbR^n)\right\}.
\end{equation}
The cost functional and the value function are given by
\begin{equation}\label{infinite-value}
\begin{aligned}
	&J_\infty^0(x;u)=\dbE\int_0^\infty
	[X(t)^\top QX(t)+2u(t)^\top SX(t)+u(t)^\top Ru(t)]dt,\\
	&V_\infty^0(x)=\inf_{u\in\mathcal U_{\mathrm{ad}}(x)}J_\infty^0(x;u).
\end{aligned}
\end{equation}

The stability of the uncontrolled system also gives the following estimate uniformly over the time horizon.
The proof can be found in Sun--Yong \cite{SunYong2024ZeroSum}.

\begin{lemma}
	\label{lem:state-estimate}
	 If $[A,C]$ is $L^2$-stable, then for
	 any $0\le t<T$, $x\in\dbR^n$, and $u\in L^2_{\dbF}(t,T;\dbR^m)$, the corresponding homogeneous controlled state satisfies
	\begin{equation*}
		\dbE\int_t^T|X(s)|^2ds+\dbE|X(T)|^2 \le \cC\left(|x|^2+\dbE\int_t^T|u(s)|^2ds\right),
	\end{equation*}
	where $\cC>0$ is independent of $t$ and $T$.
\end{lemma}

\begin{corollary}
	\label{cor:exponential-stability}
	If $[A,C]$ is $L^2$-stable, then there exist constants $\cC,\gamma>0$ such that, for every $0\le t\le s$ and $x\in\dbR^n$, the uncontrolled homogeneous state starting from $x$ at time $t$ satisfies
	\[
	\dbE|X(s;t,x)|^2\le \cC e^{-2\gamma(s-t)}|x|^2.
	\]
\end{corollary}

We denote the system  \eqref{state-homo} by $[A,C;B,D]$, and recall the following notion.
\begin{definition}
	The system $[A,C;B,D]$ is called $L^2$-stabilizable if there exists a matrix $\Theta\in\dbR^{m\times n}$ such that the closed-loop system $[A+B\Theta,C+D\Theta]$ is $L^2$-stable.
	Such a matrix $\Theta$ is called a stabilizer of $[A,C;B,D]$.
\end{definition}

For $T>0$ and $u\in L^2_{\dbF}(0,T;\dbR^m)$, let $X$ be the solution of the homogeneous finite-horizon state equation
\begin{equation*}
    \left\{
    \begin{aligned}
    &dX(t)=\bigl(AX(t)+Bu(t)\bigr)dt +\bigl(CX(t)+Du(t)\bigr)dW(t),\qquad t\in[0,T],\\
    &X(0)=0,
    \end{aligned}
    \right.
\end{equation*}
and define
\begin{equation*}
    J_T^0(0;u) = \dbE\int_0^T \left[ X(t)^\top QX(t) +2X(t)^\top S^\top u(t) +u(t)^\top Ru(t) \right]dt.
\end{equation*}

We impose the following assumptions.
\begin{taggedassumption}{(H1)}\label{ass:L2-stabilizable} The  system $[A,C;B,D]$ is $L^2$-stabilizable.
\end{taggedassumption}
\begin{taggedassumption}{(H2)}\label{ass:JC} There exists a constant $\delta>0$, independent of $T$ and $u$, such that
	\begin{equation}\label{JC}
		J_T^0(0;u)\geq\delta\dbE\int_0^T
		\bigl(|X^{0,u}(t)|^2+|u(t)|^2\bigr)\,dt,
		\qquad \forall T>0,\quad\forall u\in L^2_{\dbF}(0,T;\dbR^m).
	\end{equation}
\end{taggedassumption}

We refer to Assumption~\ref{ass:JC} as the \emph{joint coercivity} ($JC$) condition. This terminology is motivated by the following \emph{coercivity} condition, introduced by Damm and Jacob \cite{Damm2020}: There exists $\delta>0$ such that
\begin{equation}\label{C-infinite}
    J_\infty^0(0;u)\geq\delta\dbE\int_0^\infty|X^{0,u}(t)|^2\,dt,
    \qquad\forall u\in\mathcal U_{\mathrm{ad}}(0).
\end{equation}
A related condition commonly imposed in the literature is the \emph{uniform convexity} ($UC$) of the cost functional with respect to the control; see \cite{Sun-Li-Yong2016}.
The homogeneous cost functional satisfies the uniform convexity condition,   if for any $T>0$, there exists a constant $\delta_T>0$, independent of  $u$, such that
	\begin{equation}\label{UC}
		J_T^0(0;u)\ge\delta_T\dbE\int_0^T|u(t)|^2dt,
		\qquad u\in L^2_{\dbF}(0,T;\dbR^m).
	\end{equation}

\begin{remark}\label{remark2.10}
For a finite-horizon problem with a fixed $T>0$, uniform convexity and joint coercivity are equivalent, with constants that may depend on $T$.
Indeed, for a zero initial state, by the standard estimates, we have
$$\dbE\int_0^T|X(t)|^2dt\le \cC_T \dbE\int_0^T|u(t)|^2dt,$$
for some constant  $\cC_T>0$. Then,
$$J_T^0(0;u)\ge\left( {\delta_T\over 1+\cC_T}\right)\dbE\int_0^T
		\bigl(|X^{0,u}(t)|^2+|u(t)|^2\bigr)\,dt,\quad u\in L^2_{\dbF}(0,T;\dbR^m).$$
\end{remark}

% For the infinite-horizon problem, the two conditions \eqref{JC}--\eqref{UC} are also equivalent when the system is stable.

\begin{proposition}\label{prop:JC-UC}
Suppose that the uncontrolled system $[A,C]$ is $L^2$-stable.
Then \eqref{JC} is equivalent to \eqref{UC} holding with a common
constant $\delta_T=\delta_0>0$ independent of $T$.
\end{proposition}

\begin{proof}
%	We only need to show that $UC$ implies $JC$. Suppose that the $UC$ condition holds.
	It suffices to prove the reverse implication. Suppose that
	\eqref{UC} holds with a common constant $\delta_0>0$ independent of $T$.
	By Lemma~\ref{lem:state-estimate}, there exists a constant $\cC>0$, independent of $T$ and $u$, such that, for a zero initial state,
	\[
	\dbE\int_0^T|X(t)|^2dt\le\cC\dbE\int_0^T|u(t)|^2dt.
	\]
	Combining this estimate with \eqref{UC}, we obtain
	\[
	\dbE\int_0^T\bigl(|X(t)|^2+|u(t)|^2\bigr)dt
	\le(1+\cC)\dbE\int_0^T|u(t)|^2dt
	\le\frac{1+\cC}{\delta_0}J_T^0(0;u).
	\]
	Taking $\delta=\delta_0/(1+\cC)$ gives \eqref{JC}.
\end{proof}

In finite-horizon stochastic LQ theory, the standard condition requires the terminal weight matrix to be non-negative and assumes that there exists a constant $\delta_0>0$ such that
\begin{equation}\label{standard-C}
    R\ge\delta_0 I_m, \qquad Q-S^\top R^{-1}S\ge0.
\end{equation}
On every fixed finite horizon $[0,T]$, the standard condition implies the uniform convexity of the homogeneous cost functional $u\mapsto J_T^0(0;u)$, where the convexity constant may depend on $T$; this uniform convexity in turn guarantees that the corresponding Riccati equation admits a unique strongly regular solution. These results can be found in Sun--Li--Yong \cite[Proposition~3.5, Theorem~4.5 and Remark~4.6]{Sun-Li-Yong2016}. However, the standard condition does not provide a single convexity constant valid for all $T>0$ and therefore does not imply \ref{ass:JC}. The following example illustrates this point.

\begin{example}\label{example1}
The standard condition does not imply \ref{ass:JC}.
Consider the deterministic scalar system
\begin{equation}\label{example-state}
	\left\{
	\begin{aligned}
		&\dot X(t)=-X(t)+3u(t),\qquad 0\le t\le T,\\
		&X(0)=0
	\end{aligned}
	\right.
\end{equation}
with cost functional
\begin{equation*}
    J_T^0(0;u) = \int_0^T\bigl(u(t)-X(t)\bigr)^2dt.
\end{equation*}
Here $A=-1$, $B=3$, $C=D=0$, and $R=1$, $S=-1$, $Q=1$.
Thus the free system is exponentially stable and $Q-S^\top R^{-1}S=0$, so the standard condition \rf{standard-C} holds.
For each $T>0$, let $u_T(t)-X(t)=3e^{-t}$.
It follows from \eqref{example-state} that
\[
 X(t)=3\bigl(e^{2t}-e^{-t}\bigr), \qquad u_T(t)=3e^{2t}.
\]
Consequently,
\begin{equation*}
    \begin{aligned}
    \frac{J_T^0(0;u_T)}{\displaystyle\int_0^T|u_T(t)|^2dt} &=\frac{\displaystyle\int_0^T9e^{-2t}dt}{\displaystyle\int_0^T9e^{4t}dt}
    =\frac{2(1-e^{-2T})}{e^{4T}-1}\longrightarrow0, \qquad \hbox{as }T\to\infty.
    \end{aligned}
\end{equation*}
Hence no positive constant $\delta$ independent of $T$ can satisfy \ref{ass:JC}.
\end{example}

\begin{remark}
It is straightforward to show that the optimal control in Example \ref{example1} is given by \(u^*(t)=X^*(t)\), where \(X^*\) denotes the optimal state process. Note that for the initial condition \(X^*(0)=x\) with \(x\neq 0\), the optimal state process is \(X^*(t)=e^{2t}x\), \(t\in[0,T]\), which clearly fails to satisfy the turnpike property. Indeed, the standard condition \rf{standard-C} does not, in general, imply the turnpike property. In such cases, one usually needs to impose the so-called \emph{detectability condition}.
\end{remark}

The following example shows that a system may still satisfy the $JC$ condition while its uncontrolled dynamic system is unstable and its cost functional  fails to satisfy the standard condition.

\begin{example}
	Consider the system
	\[
	\left\{
	\begin{aligned}
		&\dot X(t)=X(t)+u(t),\qquad 0\le t\le T,\\
		&X(0)=0,
	\end{aligned}
	\right.
	\]
	and the cost functional
	\[
	J_T^0(0;u)=\int_0^T\bigl(6X(t)^2+6X(t)u(t)+u(t)^2\bigr)dt.
	\]
	The uncontrolled system is unstable. However, the feedback $u(t)=-2X(t)$ yields the closed-loop equation $\dot X(t)=-X(t)$, so the system is stabilizable.
	In this case, $Q=6,S=3,R=1$, and
$	Q-S^\top R^{-1}S=6-9=-3<0 ,$
	so the standard condition fails.
	For any $T>0$ and $u\in L^2(0,T)$, differentiation gives
	\[
	\frac d{dt}X(t)^2
	=2X(t)^2+2X(t)u(t)
	\ge X(t)^2-u(t)^2.
	\]
	Integrating and using $X(0)=0$, we obtain
$	\int_0^T X(t)^2dt\le\int_0^T u(t)^2dt+X(T)^2.$
	Moreover, the state equation yields
	\begin{equation*}
			J_T^0(0;u)
			=3\int_0^T\frac d{dt}X(t)^2dt+\int_0^T u(t)^2dt
			=3X(T)^2+\int_0^T u(t)^2dt
			\ge\frac12\int_0^T\bigl(X(t)^2+u(t)^2\bigr)dt.
	\end{equation*}
\end{example}

\begin{proposition}
\label{prop:strong-regularity}
Let \ref{ass:JC} hold.
Then, for every $T>0$, the Riccati equation \eqref{Riccati-T} admits a unique strongly regular solution $P_T\in C([0,T];\dbS^n)$.
Moreover,
\begin{equation}\label{prop:strong-regularity1}
    R+D^\top P_T(t)D\ge\delta I_m, \qquad 0\le t\le T.
\end{equation}
In particular, $R\geq\delta I_m$.
\end{proposition}

\begin{remark}
Condition \eqref{JC} implies \eqref{UC}
with $\delta_T\equiv\delta$.
Then the above result is given by
Sun--Li--Yong \cite[Theorem~4.5 and Remark~4.6]{Sun-Li-Yong2016} immediately.
The result $R\geq\delta I_m$ is obtained by \eqref{prop:strong-regularity1} and the fact $P_T(T)=0$.
%Moreover,  it gives unique open-loop solvability of the corresponding nonhomogeneous finite-horizon stochastic LQ problem and a  state-feedback representation.
\end{remark}

\begin{proposition}\label{prop:JC-properties}
Assumption~\ref{ass:JC} implies
	\begin{equation}\label{JC-infinite}
		J_\infty^0(0;u)\geq\delta\dbE\int_0^\infty
		\bigl(|X^{0,u}(t)|^2+|u(t)|^2\bigr)\,dt,
		\qquad \forall u\in\mathcal U_{\mathrm{ad}}(0).
	\end{equation}
\end{proposition}

%The proof of the above is clear.
This follows by letting $T\to\infty$ in \eqref{JC} for $u\in\mathcal U_{\mathrm{ad}}(0)$.
Now we return to the homogeneous infinite-horizon problem defined by \eqref{state-homo}--\eqref{admissible}--\eqref{infinite-value}, and determine $P_\infty$ and the stabilizing feedback. The theorem follows the approach of Damm--Jacob~\cite[Theorem~2]{Damm2020}, where
their diffusion coefficient does not depend on the control; we establish the result for the general case.
The proof is given in the Appendix.

Consider the algebraic Riccati equation
\begin{equation}\label{ARE}
			\begin{aligned}
				0={}&A^\top P_\infty+P_\infty A+C^\top P_\infty C+Q\\
				&-(P_\infty B+C^\top P_\infty D+S^\top)
				(R+D^\top P_\infty D)^{-1}
				(B^\top P_\infty+D^\top P_\infty C+S).
			\end{aligned}
		\end{equation}
We call $P_\infty\in\dbS^n$ a \emph{stabilizing solution} of \rf{ARE} if it satisfies this equation and the corresponding closed-loop system $[A-BK_\infty,C-DK_\infty]$ is $L^2$-stable, where \begin{equation}\label{gain}
			\begin{aligned}
				&M_\infty=R+D^\top P_\infty D>0,\q
				N_\infty=B^\top P_\infty+D^\top P_\infty C+S,
				\q K_\infty=M_\infty^{-1}N_\infty.
			\end{aligned}
		\end{equation}

\begin{theorem}\label{thm:Damm}
	Suppose that the system satisfies \ref{ass:L2-stabilizable} and that $R>0$.
	Then the following three statements are equivalent:
\begin{enumerate}[(i)]
		\item  The coercivity condition \eqref{C-infinite}
		holds for some $\delta>0$.
		\item There exists $P_\infty\in\dbS^n$ satisfying
		\rf{ARE}
		such that the system $[A-BK_\infty,C-DK_\infty]$ is mean-square exponentially stable.
 \item There exists $\delta>0$ such that
		\begin{equation}\label{JC-infinite1}
			J_\infty^0(0;u)\geq\delta\dbE\int_0^\infty
		\bigl(|X^{0,u}(t)|^2+|u(t)|^2\bigr)\,dt,
			\qquad\forall u\in\mathcal U_{\mathrm{ad}}(0).
		\end{equation}
	\end{enumerate}
	In this case, $P_\infty$ is the unique stabilizing solution of \rf{ARE}, and
	\begin{equation}\label{optimal}
		V_\infty^0(x)=x^\top P_\infty x,\qquad
		u^*(t)=-K_\infty X^*(t)
	\end{equation}
	give the optimal value and the unique optimal control for each initial state.
	The matrix $P_\infty$ may be indefinite.
\end{theorem}

\begin{corollary}\label{cor:JC-Damm}
	If \ref{ass:L2-stabilizable} and \ref{ass:JC} hold, then \eqref{ARE} admits a unique stabilizing solution, and all the conclusions of Theorem~\ref{thm:Damm} hold.
\end{corollary}

\begin{proof}
	By Proposition~\ref{prop:strong-regularity}, \eqref{prop:strong-regularity1} gives	$R\ge\delta I_m>0$.
	Moreover, \eqref{JC-infinite} in Proposition~\ref{prop:JC-properties} implies the coercivity condition \eqref{C-infinite}.
	Thus the assumptions of Theorem~\ref{thm:Damm} are satisfied, and its conclusions follow.
\end{proof}

\section{Static Optimization Problem}
\label{sec:static-distribution-optimization}

We formulate the static problem over general control distributions.
We then use the maximum-entropy property to show that its optimizer is Gaussian, so the problem can be reduced to the variables $(x,u,\Lambda)$, where $x\in\dbR^n$, $u\in\dbR^m$, and $\Lambda\in\dbS^m_{++}$.
Here $u$ is the mean of the static control distribution, and $\Lambda$ is the static exploration covariance.

We first give a static optimization problem.
Define the pointwise running part
\begin{equation*}
    L_0(x,u) = x^\top Qx+2x^\top S^\top u+u^\top Ru+2q^\top x+2r^\top u.
\end{equation*}
Define the feasible set
\begin{equation*}
    \mathcal V = \left\{(x,u)\in\dbR^n\times\dbR^m\ \middle|\ Ax+Bu+b=0\right\}.
\end{equation*}
%The diffusion contributes the $P_\infty$-weighted term to the static objective. Define
Let $P_\infty$ be the unique stabilizing solution of \eqref{ARE}.
	Define
\begin{equation}\label{F0-1}
    F_0(x,u) = L_0(x,u) + (Cx+Du+\sigma)^\top P_\infty(Cx+Du+\sigma).
\end{equation}
The classical stochastic LQ static problem is
\begin{equation}
    \label{static-problem}
    \begin{aligned}
    \min_{(x,u)\in\mathcal V}\quad &F_0(x,u).
    \end{aligned}
\end{equation}

\begin{proposition}
\label{prop:static-KKT}
Let \ref{ass:L2-stabilizable}--\ref{ass:JC} hold.
Then the static problem \eqref{static-problem} has a unique solution, denoted by $(x_*,u_*)$.
Moreover, there exists a Lagrange multiplier $\zeta^*\in\dbR^n$ satisfying the Karush--Kuhn--Tucker (KKT) conditions
\begin{equation}
    \label{KKT-system}
    \left\{
    \begin{aligned}
    &Ax_*+Bu_*+b=0,\\
    &Qx_*+S^\top u_*+q+C^\top P_\infty(Cx_*+Du_*+\sigma)+A^\top \zeta^*=0,\\
    &Sx_*+Ru_*+r+D^\top P_\infty(Cx_*+Du_*+\sigma)+B^\top \zeta^*=0.
    \end{aligned}
    \right.
\end{equation}
Conversely, if a triple $(x,u,\zeta)$ satisfies the system obtained from \eqref{KKT-system} by replacing $(x_*,u_*,\zeta^*)$ with $(x,u,\zeta)$, then $(x,u)$ is the unique minimizer, that is, $(x_*,u_*)=(x,u)$.
\end{proposition}

\begin{remark}
In contrast to  Sun--Wang--Yong \cite{SunWangYong2022}, the matrices \(Q\) and \(P_\infty\) in \rf{F0-1} can be indefinite. Therefore, we need to introduce new ideas and perform a careful convexity analysis of \(F_0\).
\end{remark}

\begin{proof}
By Corollary~\ref{cor:JC-Damm} and Corollary~\ref{cor:Hurwitz},
$A-BK_\infty$ is Hurwitz and hence invertible.
Let
\begin{equation*}
    x^0=-(A-BK_\infty)^{-1}b, \qquad u^0=-K_\infty x^0.
\end{equation*}
Then
\begin{equation*}
    Ax^0+Bu^0+b=(A-BK_\infty)x^0+b=0.
\end{equation*}
Hence $(x^0,u^0)\in\mathcal V$. Moreover,
\begin{equation*}
    [A\ \ B]
    \begin{pmatrix}I_n\\-K_\infty\end{pmatrix}
    =A-BK_\infty
\end{equation*}
is invertible, so $[A\ \ B]$ has full row rank.

Fix $(x,u)\in\mathcal V$.
For any $(\xi,v)\in\dbR^n\times\dbR^m$,
\begin{equation}
    \label{feasible-directions}
    (x+\xi,u+v)\in\mathcal V \quad\Longleftrightarrow\quad A\xi+Bv=0.
\end{equation}
Let $\mathcal T:=\{(\xi,v)\in\dbR^n\times\dbR^m\mid A\xi+Bv=0\}$ be the tangent space.
%Note that
%\begin{equation*}
 %   \begin{aligned}
 %   F_0(x,u)={}&x^\top Qx+2x^\top S^\top u+u^\top Ru+2q^\top x+2r^\top u
 %   +(Cx+Du+\sigma)^\top P_\infty(Cx+Du+\sigma).
 %   \end{aligned}
%\end{equation*}
%The Taylor expansion of the quadratic function $F_0$ gives
The Taylor expansion of the quadratic function $F_0$
	in \eqref{F0-1} gives
\begin{equation}
    \label{quadratic-expansion}
    \begin{aligned}
    F_0(x+\xi,u+v)-F_0(x,u) ={}&2f_x(x,u)^\top\xi+2f_u(x,u)^\top v +f_{xu}(\xi,v).
    \end{aligned}
\end{equation}
Here $2f_x$ and $2f_u$ are the corresponding gradient components, with
\begin{equation*}
    \begin{aligned}
    &f_x(x,u) =Qx+S^\top u+q+C^\top P_\infty(Cx+Du+\sigma),\\
    &f_u(x,u) =Sx+Ru+r+D^\top P_\infty(Cx+Du+\sigma).
    \end{aligned}
\end{equation*}
The second-order term is
\begin{equation*}
    \begin{aligned}
    f_{xu}(\xi,v)={}&\xi^\top(Q+C^\top P_\infty C)\xi +2\xi^\top(S^\top+C^\top P_\infty D)v
    +v^\top(R+D^\top P_\infty D)v.
    \end{aligned}
\end{equation*}
%Set
%\begin{equation*}
 %   \begin{aligned}
 %   M_\infty&:=R+D^\top P_\infty D,\qq
 %   N_\infty:=B^\top P_\infty+D^\top P_\infty C+S,\qq
 %   K_\infty:=M_\infty^{-1}N_\infty.
 %   \end{aligned}
%\end{equation*}
%For $(\xi,v)\in\mathcal T$, the algebraic Riccati equation \eqref{ARE}, the symmetry of $P_\infty$, the tangent constraint %\eqref{feasible-directions}, and the definitions of $N_\infty$ and $K_\infty$ yield
For $(\xi,v)\in\mathcal T$, the algebraic Riccati equation
	\eqref{ARE}, the symmetry of $P_\infty$, the tangent constraint
	\eqref{feasible-directions}, and the definitions in \eqref{gain} yield
{
\begin{equation*}
\begin{aligned}
f_{xu}(\xi,v)
%&=\xi^\top(Q+C^\top P_\infty C)\xi+2\xi^\top\!(S^\top+C^\top P_\infty D)\!v+v^\top M_\infty v
%&&\text{(by the definition of $f_{xu}$)}\\
%&=\xi^\top(-P_\infty A-A^\top P_\infty+N_\infty^\top M_\infty^{-1}N_\infty)\xi+2\xi^\top\!(S^\top+C^\top P_\infty D)\!v+v^\top M_\infty v
%&&\text{(by \eqref{ARE})}\\
%&=-2\xi^\top P_\infty A\xi+\xi^\top N_\infty^\top M_\infty^{-1}N_\infty\xi+2\xi^\top\!(S^\top+C^\top P_\infty D)\!v+v^\top M_\infty v
%&&\text{(by the symmetry of $P_\infty$)}\\
%&=2\xi^\top P_\infty Bv+\xi^\top N_\infty^\top M_\infty^{-1}N_\infty\xi+2\xi^\top\!(S^\top+C^\top P_\infty D)\!v+v^\top M_\infty v
%&&\text{(by \eqref{feasible-directions})}\\
=\xi^\top N_\infty^\top M_\infty^{-1}N_\infty\xi+2\xi^\top N_\infty^\top v+v^\top M_\infty v
%&&\text{(by the definition of $N_\infty$)}\\
%&=(v+M_\infty^{-1}N_\infty\xi)^\top M_\infty(v+M_\infty^{-1}N_\infty\xi)
%&&\text{(by completing the square)}\\
=(v+K_\infty\xi)^\top M_\infty(v+K_\infty\xi).
%&&\text{(by $K_\infty=M_\infty^{-1}N_\infty$)}.
\end{aligned}
\end{equation*}
}
Since $M_\infty>0$, we have $f_{xu}(\xi,v)\geq0$.
If equality holds, then $v=-K_\infty\xi$. Together with \eqref{feasible-directions}, this gives
\begin{equation*}
    (A-BK_\infty)\xi=0.
\end{equation*}
The invertibility of $A-BK_\infty$ implies $\xi=0$ and then $v=0$.
Hence $f_{xu}(\xi,v)>0$ for all $(\xi,v)\in\mathcal T\setminus\{(0,0)\}$.

The unit sphere $\mathbb S_{\mathcal T}:=\{(\xi,v)\in\mathcal T\mid |\xi|^2+|v|^2=1\}$ is compact, and therefore
$c:=\min_{(\xi,v)\in\mathbb S_{\mathcal T}}f_{xu}(\xi,v)>0$.
For any nonzero $(\xi,v)\in\mathcal T$, let $r=\sqrt{|\xi|^2+|v|^2}$.
The quadratic homogeneity of $f_{xu}$ gives
\begin{equation}
    \label{quadratic-bound}
    \begin{aligned}
    f_{xu}(\xi,v) &=r^2f_{xu}\left(\frac{\xi}{r},\frac{v}{r}\right)
    %&&\text{(by quadratic homogeneity)}\\
    \geq cr^2
    =c\bigl(|\xi|^2+|v|^2\bigr).
    \end{aligned}
\end{equation}
The assertion is immediate for $(\xi,v)=(0,0)$.

Every $(x,u)\in\mathcal V$ can be written as $(x,u)=(x^0,u^0)+(\xi,v)$ with $(\xi,v)\in\mathcal T$.
By \eqref{quadratic-expansion} and \eqref{quadratic-bound}, there exists a constant $\cC>0$, independent of $(\xi,v)$, such that
{
\begin{equation*}
\begin{aligned}
F_0(x^0+\xi,u^0+v)
%&=F_0(x^0,u^0)+2f_x(x^0,u^0)^\top\xi+2f_u(x^0,u^0)^\top v+f_{xu}(\xi,v)
%&&\text{(by \eqref{quadratic-expansion})}\\
%&\geq f_{xu}(\xi,v)-\cC\bigl(|\xi|+|v|\bigr)-\cC\\
&\geq c\bigl(|\xi|^2+|v|^2\bigr)-\cC\bigl(|\xi|+|v|\bigr)-\cC.
%&&\text{(by \eqref{quadratic-bound})}.
\end{aligned}
\end{equation*}
}
%Thus $F_0$ is coercive on $\mathcal V$. This means that $F_0(x,u)\to\infty$ whenever $(x,u)\in\mathcal V$ and %$|x|+|u|\to\infty$. Consequently, every minimizing sequence for $F_0$ on $\mathcal V$ is bounded and hence has a %convergent subsequence in the finite-dimensional space. Since $\mathcal V$ is a closed affine set and $F_0$ is continuous, %the limit of this subsequence belongs to $\mathcal V$ and attains the minimum of $F_0$.
Thus $F_0$ is coercive on the closed affine set $\mathcal V$. Since $F_0$ is continuous, it attains its minimum on $\mathcal V$.
Moreover, \eqref{quadratic-bound} gives $f_{xu}(\xi,v)>0$ for every nonzero feasible direction $(\xi,v)\in\mathcal T$.
Hence $F_0$ is strictly convex on $\mathcal V$, so it has at most one minimizer.
Therefore the static problem \eqref{static-problem} admits a unique minimizer $(x_*,u_*)$.

Since $[A\ \ B]$ has full row rank, the finite-dimensional Lagrange multiplier theorem for equality constraints (see Yong \cite{Yong2018}) yields a $\zeta^*\in\dbR^n$ such that
\begin{equation*}
    \nabla F_0(x_*,u_*)+2[A\ \ B]^\top \zeta^*=0.
\end{equation*}
%Since $\nabla_xF_0(x,u)=2f_x(x,u)$ and $\nabla_uF_0(x,u)=2f_u(x,u)$, writing the two block components gives %\eqref{KKT-system} as follows:
%\begin{equation*}
%    \left\{
%    \begin{aligned}
%    &Ax_*+Bu_*+b=0,\\
%    &Qx_*+S^\top u_*+q+C^\top P_\infty(Cx_*+Du_*+\sigma)+A^\top \zeta^*=0,\\
%    &Sx_*+Ru_*+r+D^\top P_\infty(Cx_*+Du_*+\sigma)+B^\top \zeta^*=0.
%    \end{aligned}
%    \right.
%\end{equation*}
Since $\nabla_xF_0=2f_x$ and $\nabla_uF_0=2f_u$, this condition, together with feasibility, gives \eqref{KKT-system}.

Conversely, let $(x,u,\zeta)$ satisfy \eqref{KKT-system} and let $(\xi,v)\in\mathcal T$.
The KKT equations give
\begin{equation*}
    f_x(x,u)=-A^\top \zeta, \qquad f_u(x,u)=-B^\top \zeta.
\end{equation*}
\begin{equation*}
    f_x(x,u)^\top\xi+f_u(x,u)^\top v =-\zeta^\top(A\xi+Bv)=0.
\end{equation*}
By \eqref{quadratic-expansion} and \eqref{quadratic-bound}, for every nonzero $(\xi,v)\in\mathcal T$,
\begin{equation*}
    \begin{aligned}
    F_0(x+\xi,u+v)-F_0(x,u) &=f_{xu}(\xi,v)
    %&&\text{(by the KKT equations and $A\xi+Bv=0$)}\\
    \geq c\bigl(|\xi|^2+|v|^2\bigr)
    %&&\text{(by \eqref{quadratic-bound})}\\
    >0.
    %&&\text{(since $(\xi,v)\ne(0,0)$)}.
    \end{aligned}
\end{equation*}
Thus $(x,u)$ is the unique minimizer, and therefore $(x,u)=(x_*,u_*)$.
\end{proof}

For a probability distribution $\pi$ on $\dbR^m$ with finite second moment, set
\begin{equation*}
    \bar u_\pi=\int_{\dbR^m}u\,\pi(du), \qquad \Lambda_\pi=\int_{\dbR^m}(u-\bar u_\pi)(u-\bar u_\pi)^\top\pi(du).
\end{equation*}
Denote
\begin{equation}
    \label{distribution-objective}
    \begin{aligned}
    F_\lambda(x,\pi)&:=\int_{\dbR^m}F_0(x,u)\,\pi(du)-\lambda\Ent(\pi).
    \end{aligned}
\end{equation}
The exploratory static problem over general distributions is
\begin{equation}
    \label{distribution-problem}
    \begin{aligned}
    \min_{x,\pi}\quad  F_\lambda(x,\pi)\qq
    \text{s.t.}\quad  Ax+B\bar u_\pi+b=0.
    \end{aligned}
\end{equation}
Here $\pi$ ranges over probability distributions that are absolutely continuous with respect to Lebesgue measure and have finite second moment and finite differential entropy.

Expanding $F_0(x,u)$ with respect to the control variable gives
\begin{equation*}
    F_0(x,u) = c_\infty(x)+2g_\infty(x)^\top u+u^\top M_\infty u,
\end{equation*}
where
\begin{equation*}
    \begin{aligned}
    &g_\infty(x) =Sx+r+D^\top P_\infty(Cx+\sigma),\\
    &c_\infty(x) =x^\top Qx+2q^\top x+(Cx+\sigma)^\top P_\infty(Cx+\sigma).
    \end{aligned}
\end{equation*}
%Note that
%\begin{equation*}
    %u^\top M_\infty u=\tr(M_\infty uu^\top),
%\end{equation*}
%and
%\begin{equation*}
    %\int_{\dbR^m}\tr(M_\infty uu^\top)\,\pi(du) = \tr\left(M_\infty\int_{\dbR^m}uu^\top\pi(du)\right).
%\end{equation*}
Taking expectation under a general distribution $\pi$ and using the second-moment decomposition $\int_{\dbR^m}uu^\top\,\pi(du)=\bar u_\pi\bar u_\pi^\top+\Lambda_\pi$, we obtain
\begin{equation*}
\begin{aligned}
    \int_{\dbR^m}F_0(x,u)\,\pi(du)
    %&=\int_{\dbR^m}\left[c_\infty(x)+2g_\infty(x)^\top u+u^\top M_\infty u\right]\pi(du)
    %\\
    %&=c_\infty(x)+2g_\infty(x)^\top\bar u_\pi
    %+\int_{\dbR^m}u^\top M_\infty u\,\pi(du)
   %\\
    %&=c_\infty(x)+2g_\infty(x)^\top\bar u_\pi
    %+\tr\left(M_\infty\int_{\dbR^m}uu^\top\pi(du)\right)\\
    &=c_\infty(x)+2g_\infty(x)^\top\bar u_\pi
    +\bar u_\pi^\top M_\infty\bar u_\pi+\tr(M_\infty\Lambda_\pi)\\
    &=F_0(x,\bar u_\pi)+\tr(M_\infty\Lambda_\pi).
\end{aligned}
\end{equation*}
%
%Thus the objective in \rf{distribution-problem} can be written as
By \eqref{distribution-objective}, the objective in
	\eqref{distribution-problem} can be written as
\begin{equation}
    \label{distribution-decomposition}
    F_\lambda(x,\pi)=F_0(x,\bar u_\pi)+\tr(M_\infty\Lambda_\pi)-\lambda\Ent(\pi).
\end{equation}
Except for the entropy term, the objective only depends on the first moment and the centered second moment of the control distribution.

The following is taken from Cover--Thomas \cite{CoverThomas2006}.

\begin{lemma}
\label{lem:maximum-entropy}
Let $\pi(du)=p(u)du$ be a probability distribution on $\dbR^m$ with finite second moment and finite entropy.
Let $u\in\dbR^m$ and $\Lambda\in\dbS^m_{++}$ be its mean and covariance, respectively.
Then
\begin{equation*}
    %\label{maximum-entropy}
    \Ent(\pi) \le \Ent\big(\cNcal(u,\Lambda)\big) = \frac12\log\big((2\pi e)^m\det\Lambda\big).
\end{equation*}
Equality holds if and only if $\pi=\cNcal(u,\Lambda)$.
\end{lemma}

Take any feasible pair $(x,\pi)$, and set $u=\bar u_\pi$, $\Lambda=\Lambda_\pi$, and $\pi_G=\cNcal(u,\Lambda)$.
Since $\pi_G$ and $\pi$ have the same mean and covariance, the drift constraint remains unchanged:
$Ax+B\bar u_{\pi_G}+b=Ax+B\bar u_\pi+b$.
Also, all terms in \eqref{distribution-decomposition} except the entropy term are unchanged.
By Lemma~\ref{lem:maximum-entropy}, $\Ent(\pi_G)\ge\Ent(\pi)$, and hence
\begin{equation*}
    F_\lambda(x,\pi_G) \le F_\lambda(x,\pi).
\end{equation*}
If $\pi$ is not Gaussian, the above inequality is strict.
Thus the general distribution problem \eqref{distribution-problem} is equivalent to the following Gaussian moment problem.

\begin{taggedthm}{Problem $(O)$}
\begin{equation}
    \label{Gaussian-problem}
    \begin{aligned}
    \min_{x,u,\Lambda>0}\quad  \mathcal F_\lambda(x,u,\Lambda)\qq
    \text{s.t.}\quad  Ax+Bu+b=0,
    \end{aligned}
\end{equation}
where $\Lambda>0$ means $\Lambda\in\dbS^m_{++}$, and
\begin{equation}
    \label{objective-Gaussian}
    \begin{aligned}
    \mathcal F_\lambda(x,u,\Lambda) &:= F_0(x,u) + \tr(M_\infty\Lambda) - \frac{\lambda}{2}\log\big((2\pi e)^m\det\Lambda\big).
    \end{aligned}
\end{equation}
\end{taggedthm}

We separate the terms in \eqref{Gaussian-problem} that contain $\Lambda$.
%Since
%\[
    %-\frac{\lambda}{2}\log\big((2\pi e)^m\det\Lambda\big) = -\frac{\lambda}{2}\log\det\Lambda - \frac{\lambda}{2}m\log(2\pi e),
%\]
%The objective in \rf{Gaussian-problem} can be written as
By \eqref{objective-Gaussian}, we can write
\begin{equation}
    \label{objective-decomposition}
    \begin{aligned}
    \mathcal F_\lambda(x,u,\Lambda)
    = F_0(x,u) + G_\lambda(\Lambda) - \frac{\lambda}{2}m\log(2\pi e),
    \end{aligned}
\end{equation}
where
\begin{equation} \label{objective-decomposition1-G}
    G_\lambda(\Lambda) = \tr(M_\infty\Lambda) - \frac{\lambda}{2}\log\det\Lambda.
\end{equation}
The last term of \eqref{objective-decomposition} only depends on $\lambda$ and the dimension $m$, and does not affect the location of the optimizer.
Hence the state--mean variables $(x,u)$ and the exploration covariance variable $\Lambda$ can be optimized separately.

\begin{proposition}
\label{prop:optimal-covariance}
The covariance part $G_\lambda$, defined by \rf{objective-decomposition1-G}, has a unique minimizer on $\dbS^m_{++}$, and this minimizer is
\begin{equation}
    \label{optimal-covariance}
        \Lambda_*^\lambda=\Lambda_\infty^\lambda
    :=\frac{\lambda}{2}M_\infty^{-1}
    =\frac{\lambda}{2}\big(R+D^\top P_\infty D\big)^{-1}.
\end{equation}
\end{proposition}

\begin{proof}
Let $H=H^\top$ be any symmetric perturbation.
By linearity of the trace and the differentiation formulas for matrix-variable functions (see \cite[Section~2.2.2]{LiuHuLiWen2020}), setting
$f(\Lambda)=\log\det\Lambda$ gives
\[
    \nabla_\Lambda\tr(M_\infty\Lambda)=M_\infty, \qquad \nabla f(\Lambda)=(\Lambda^{-1})^\top=\Lambda^{-1}, \qquad \nabla^2 f(\Lambda)H=-\Lambda^{-1}H\Lambda^{-1}.
\]
It follows that $\nabla G_\lambda(\Lambda)=M_\infty-\frac{\lambda}{2}\Lambda^{-1}$.
The first-order condition for an interior minimizer is
$M_\infty-\frac{\lambda}{2}\Lambda^{-1}=0$, which gives \eqref{optimal-covariance}; since $M_\infty>0$, this stationary point belongs to $\dbS^m_{++}$.
Moreover, for every symmetric matrix $H\ne0$,
\[
    \begin{aligned}
    \left\langle\nabla^2G_\lambda(\Lambda)H,H\right\rangle =\frac{\lambda}{2}\tr \bigl(\Lambda^{-1}H\Lambda^{-1}H\bigr)
    =\frac{\lambda}{2} \left\|\Lambda^{-1/2}H\Lambda^{-1/2}\right\|_F^2 >0.
    \end{aligned}
\]
Thus $G_\lambda$ is strictly convex on $\dbS^m_{++}$, and the stationary point in \eqref{optimal-covariance} is its unique global minimizer.
\end{proof}

Since the optimized covariance part only contributes a constant independent of $(x,u)$, minimizing with respect to $(x,u)$ is exactly the classical stochastic static problem \eqref{static-problem}.
Therefore the static mean center of the exploratory problem is unchanged:
\begin{equation*}
    (x_*^\lambda,u_*^\lambda)=(x_*,u_*).
\end{equation*}
%Hence the unique optimizer of the exploratory static distribution problem is
Hence the unique optimizer of the Gaussian moment problem
\eqref{Gaussian-problem} is
\begin{equation*}
    (x_*,u_*,\Lambda_*^\lambda),
\end{equation*}
and the corresponding static Gaussian control distribution is
\begin{equation*}
    \pi_*^\lambda(\cdot) = \cNcal(u_*,\Lambda_*^\lambda).
\end{equation*}

\begin{proposition}
Under \ref{ass:L2-stabilizable}--\ref{ass:JC}, the Problem  $(O)$ defined by \rf{Gaussian-problem} has a unique optimizer $(x_*,u_*,\Lambda_*^\lambda)$, where $(x_*,u_*)$ is the unique solution of the classical stochastic static problem \eqref{static-problem}, and
$\Lambda_*^\lambda=\frac{\lambda}{2}(R+D^\top P_\infty D)^{-1}$.
Equivalently, the static exploratory control distribution of \rf{distribution-problem} is given by $\cNcal(u_*,\Lambda_*^\lambda)$.
\end{proposition}

\begin{proof}
The separation formula \eqref{objective-decomposition} shows that, up to an additive constant, the minimization over $(x,u)$ and the minimization over $\Lambda$ are independent.
The $(x,u)$ part is exactly \eqref{static-problem}, whose unique solution is $(x_*,u_*)$ by Proposition~\ref{prop:static-KKT}.
The covariance part has the unique solution \eqref{optimal-covariance} by Proposition~\ref{prop:optimal-covariance}.
The proof is complete.
\end{proof}

\section{Convergence of Riccati Equations}
\label{sec:riccati-parameter-convergence}

This section studies the convergence of finite-horizon parameters as $T\to\infty$.
The Gaussian feedback law in Proposition~\ref{prop:exploratory-solution}
is determined by the same Riccati equation \eqref{Riccati-T}
as in the corresponding classical stochastic LQ problem.
In this section, the constants $\cC,\gamma>0$ may change from line to line, but they are always independent of the time horizon $T$.
Unless otherwise stated, the exploration weight $\lambda>0$ is fixed.

Recall that $P_T$ is the unique strongly regular solution of
\eqref{Riccati-T}, and $P_\infty$ is the unique stabilizing solution
of \eqref{ARE}.
%Let $P_\infty$ denote the stabilizing solution of
%\eqref{ARE}, characterized in Theorem~\ref{thm:Damm}.
Similar to Sun--Wang--Yong \cite{SunWangYong2022},  we next establish an estimate for the difference between $P_T$ and $P_\infty$. However, since the cost functional is indefinite,
the method employed in Sun--Wang--Yong \cite{SunWangYong2022} cannot be applied here.
Motivated by Sun--Yong \cite{SunYong2024ZeroSum}, we overcome this difficulty via a Hilbert space approach.

\begin{lemma}
\label{lem:riccati-convergence}
Suppose that the uncontrolled system $[A,C]$ is $L^2$-stable
and that uniform convexity \eqref{UC} holds with a common constant
$\delta_T=\delta_0>0$ independent of $T$.
Let $P_T$ be the unique strongly regular solution of \eqref{Riccati-T},
and let $P_\infty$ be the unique stabilizing solution of \eqref{ARE}.
Then there exist constants $\cC,\gamma>0$, independent of $T$, such that
\begin{equation}
    \label{Riccati-convergence}
    \|P_T(t)-P_\infty\| \le \cC e^{-\gamma(T-t)}, \qquad 0\le t\le T.
\end{equation}
Moreover, $P_T(t)$ is uniformly bounded with respect to both $T$ and $t$.
\end{lemma}

\begin{proof}
%Since $[A,C]$ is $L^2$-stable, Proposition~\ref{prop:JC-UC} and \eqref{UC} imply \eqref{JC}.
%Thus Proposition~\ref{prop:strong-regularity} and
%Corollary~\ref{cor:JC-Damm} ensure the existence of
%$P_T$ and the stabilizing solution $P_\infty$.
By Proposition~\ref{prop:JC-UC}, the assumed horizon-independent
	lower bound implies \eqref{JC} with a constant $\delta>0$
	independent of $T$. We use this joint coercivity constant below.
	Proposition~\ref{prop:strong-regularity} and
	Corollary~\ref{cor:JC-Damm} then ensure the existence of $P_T$
	and the stabilizing solution $P_\infty$.

Fix $0\le t<T$, and let
\[
    \mathcal H_{t,T}=L^2_{\dbF}(t,T;\dbR^m), \qquad \|u\|_{\mathcal H_{t,T}}^2 =\dbE\int_t^T|u(s)|^2ds .
\]
Consider the following homogeneous controlled state equation
\begin{equation}\label{state-finite}
    \left\{
    \begin{aligned}
    &dX(s) = \bigl(AX(s)+Bu(s)\bigr)ds + \bigl(CX(s)+Du(s)\bigr)dW(s),\qquad s\in[t,T],\\
    &X(t)=x.
    \end{aligned}
    \right.
\end{equation}
Denote the corresponding finite-horizon cost functional by
\[
    J_{t,T}^0(x;u)=\dbE\int_t^T\bigl[X(s)^\top QX(s)+2X(s)^\top S^\top u(s)+u(s)^\top Ru(s)\bigr]ds .
\]
By classical LQ theory
	\cite{Sun-Li-Yong2016},
	the corresponding value function is
	\[
	V_T^0(t,x):=\inf_{u\in\mathcal H_{t,T}}J_{t,T}^0(x;u)
	=x^\top P_T(t)x.
	\]

We denote the solution of the state equation \rf{state-finite} with initial pair $(t,x)$ and control $u$ by $X^{t,x;u}$.
Define the zero-initial-state trajectory operator and the terminal-state operator by
\[
    \begin{aligned}
    \mathcal L_{t,T}u&=X^{t,0;u}(\cdot),\q &\mathcal L_{t,T}&:\mathcal H_{t,T} \longrightarrow L^2_{\dbF}(t,T;\dbR^n),\\
    \widehat{\mathcal L}_{t,T}u&=X^{t,0;u}(T),\q &\widehat{\mathcal L}_{t,T}&:\mathcal H_{t,T} \longrightarrow L^2_{\mathcal F_T}(\Omega;\dbR^n).
    \end{aligned}
\]
By Lemma~\ref{lem:state-estimate}, there exists a constant $\cC>0$, independent of the time horizon $[t,T]$, such that
\begin{equation}
    \label{operator-bounds}
    \|\mathcal L_{t,T}\|+\|\widehat{\mathcal L}_{t,T}\|\le \cC.
\end{equation}

By linearity and uniqueness of the state equation,
$X^{t,x;u}=X^{t,x;0}+\mathcal L_{t,T}u$.
Thus the initial state contributes only to the affine terms below, whereas the quadratic part in $u$ is exactly the part determined by the zero-initial-state operator.
Pulling the state terms back to the control space, we can write $J_{t,T}^0(x;\cdot)$ uniquely as
\begin{equation}
    \label{control-quadratic}
    J_{t,T}^0(x;u) =\langle\mathfrak M_{t,T}u,u\rangle +2\langle h_{t,T}(x),u\rangle+c_{t,T}(x),
\end{equation}
where $\mathfrak M_{t,T}$ is a bounded self-adjoint operator and
\begin{equation*}
    \mathfrak M_{t,T} =\mathcal L_{t,T}^*Q\mathcal L_{t,T} +\mathcal L_{t,T}^*S^\top +S\mathcal L_{t,T}+R.
\end{equation*}
By time homogeneity and the joint coercivity condition
\eqref{JC} established above, we have
\begin{equation}
    \label{operator-positivity}
    \mathfrak M_{t,T}\ge\delta I, \qquad \|\mathfrak M_{t,T}^{-1}\|\le\delta^{-1}.
\end{equation}
Consequently, the finite-horizon problem with zero terminal cost admits a unique optimal control $u_T$.
Denote its optimal state by $X_T:=X^{t,x;u_T}$.
The first-order optimality condition is
\begin{equation}
    \label{finite-optimality}
    \mathfrak M_{t,T}u_T+h_{t,T}(x)=0.
\end{equation}

We next compare this problem with the infinite-horizon problem on the same control space.
Let $u_\infty=-K_\infty X_\infty$ be the restriction to $[t,T]$ of the infinite-horizon optimal control.
By the dynamic programming principle for the infinite-horizon problem and the identity $V_\infty^0(x)=x^\top P_\infty x$,
\begin{equation}
    \label{dynamic-programming}
    x^\top P_\infty x =\inf_{u\in\mathcal H_{t,T}} \left\{ J_{t,T}^0(x;u) +\dbE\bigl[X(T)^\top P_\infty X(T)\bigr] \right\},
\end{equation}
and the minimizer is precisely $u_\infty$.
Taking the Fr\'echet derivative of the terminal quadratic term $\dbE[X(T)^\top P_\infty X(T)]$ in \eqref{dynamic-programming}, we obtain, for every $v\in\mathcal H_{t,T}$,
\begin{equation*}
    \begin{aligned}
    D_u\!\left\{ \dbE\bigl[X(T)^\top P_\infty X(T)\bigr] \right\}(u_\infty)[v]
    %&=\left.\frac{d}{d\varepsilon}\right|_{\varepsilon=0} \dbE\!\left[\bigl(X_\infty(T)+\varepsilon\widehat{\mathcal L}_{t,T}v\bigr)^\top
    %P_\infty\bigl(X_\infty(T)+\varepsilon\widehat{\mathcal L}_{t,T}v\bigr)\right]\\
    =2\dbE\!\left[X_\infty(T)^\top P_\infty\widehat{\mathcal L}_{t,T}v\right]
    %&&\text{(by symmetry of $P_\infty$)}\\
    %&=2\left\langle P_\infty X_\infty(T),\widehat{\mathcal L}_{t,T}v\right\rangle_{L^2_{\mathcal F_T}} &&\text{(by the terminal-space inner product)}\\
    =2\left\langle \widehat{\mathcal L}_{t,T}^*\bigl(P_\infty X_\infty(T)\bigr),v \right\rangle_{\mathcal H_{t,T}}.
    %&&\text{(by the definition of the adjoint)}.
    \end{aligned}
\end{equation*}
On the other hand, \eqref{control-quadratic} gives the derivative of the running-cost term $J_{t,T}^0(x;u)$:
\[
    D_uJ_{t,T}^0(x;u_\infty)[v] =2\langle\mathfrak M_{t,T}u_\infty+h_{t,T}(x),v\rangle_{\mathcal H_{t,T}}, \qquad v\in\mathcal H_{t,T}.
\]
Therefore, the minimizer $u_\infty$ satisfies the  first-order optimality condition for \eqref{dynamic-programming}
\begin{equation}
    \label{infinite-optimality}
    \mathfrak M_{t,T}u_\infty+h_{t,T}(x) +\widehat{\mathcal L}_{t,T}^* \bigl(P_\infty X_\infty(T)\bigr)=0.
\end{equation}
%Subtracting \eqref{finite-optimality} from
%\eqref{infinite-optimality} gives the identity and norm estimates
Subtracting \eqref{finite-optimality} from
	\eqref{infinite-optimality} and using \eqref{operator-positivity},
	we obtain
\begin{equation}
    \label{control-error}
    \begin{aligned}
    u_T-u_\infty &= \mathfrak M_{t,T}^{-1} \widehat{\mathcal L}_{t,T}^* \bigl(P_\infty X_\infty(T)\bigr),\\
    %\|u_T-u_\infty\|_{\mathcal H_{t,T}} &\le \|\mathfrak M_{t,T}^{-1}\| \|\widehat{\mathcal L}_{t,T}\| \|P_\infty\| \bigl(\dbE|X_\infty(T)|^2\bigr)^{1/2}\\
    \|u_T-u_\infty\|_{\mathcal H_{t,T}} &\le \delta^{-1}\|\widehat{\mathcal L}_{t,T}\| \|P_\infty\| \bigl(\dbE|X_\infty(T)|^2\bigr)^{1/2}.
    %&&\text{(by \eqref{operator-positivity})}.
    \end{aligned}
\end{equation}
By Corollary~\ref{cor:JC-Damm}, the closed-loop system $[A-BK_\infty,C-DK_\infty]$ is $L^2$-stable. Applying Corollary~\ref{cor:exponential-stability} to this closed-loop system with initial time $t$ and terminal time $T$, we obtain
\begin{equation}
    \label{state-decay}
    \dbE|X_\infty(T)|^2 \le \cC e^{-2\gamma(T-t)}|x|^2.
\end{equation}
It follows from \eqref{control-error} that
\begin{equation}
    \label{error-decay}
    \|u_T-u_\infty\|_{\mathcal H_{t,T}} \le \cC e^{-\gamma(T-t)}|x|.
\end{equation}
The two optimal states have the same initial condition, and therefore
\begin{equation*}
    X_T(T)-X_\infty(T) =\widehat{\mathcal L}_{t,T}(u_T-u_\infty).
\end{equation*}
Combining \eqref{operator-bounds}, \eqref{state-decay}, and \eqref{error-decay}, we obtain
\begin{equation}
    \label{terminal-decay}
    \dbE|X_T(T)|^2 \le \cC e^{-2\gamma(T-t)}|x|^2.
\end{equation}

It remains to compare the two value functions.
For the first comparison, take the infinite-horizon optimal control on $[t,\infty)$ and restrict it to $[t,T]$.
We continue to denote this restriction by $u_\infty$.
It is admissible for the finite-horizon problem with zero terminal cost.
Therefore,
\begin{equation*}
    \begin{aligned}
    V_T^0(t,x)-V_\infty^0(x) &\le J_{t,T}^0(x;u_\infty)-V_\infty^0(x)
    %&&\text{(by finite-horizon optimality)}\\
    = -\dbE\bigl[X_\infty(T)^\top P_\infty X_\infty(T)\bigr]
    %&&\text{(by \eqref{dynamic-programming})}\\
    \le \cC e^{-2\gamma(T-t)}|x|^2.
    %&&\text{(by \eqref{state-decay})}.
    \end{aligned}
\end{equation*}
%For the reverse comparison, the dynamic programming identity \eqref{dynamic-programming} reads
%\[
%    x^\top P_\infty x =\inf_{u\in\mathcal H_{t,T}} \left\{ J_{t,T}^0(x;u) +\dbE\bigl[X(T)^\top P_\infty X(T)\bigr] \right\}.
%\]
%Taking the finite-horizon optimal control $u_T$ as a competitor on its right-hand side and using %$J_{t,T}^0(x;u_T)=V_T^0(t,x)$, we obtain
Conversely, using $u_T$ as a competitor in \eqref{dynamic-programming} and the identity $J_{t,T}^0(x;u_T)=V_T^0(t,x)$, we obtain
\[
    \begin{aligned}
    V_\infty^0(x) &\le J_{t,T}^0(x;u_T) +\dbE\bigl[X_T(T)^\top P_\infty X_T(T)\bigr]
    =V_T^0(t,x) +\dbE\bigl[X_T(T)^\top P_\infty X_T(T)\bigr].
    %&&\text{(by finite-horizon optimality)}.
    \end{aligned}
\]
Using \eqref{terminal-decay}, we obtain
\begin{equation*}
	\begin{aligned}
		V_\infty^0(x)-V_T^0(t,x) &\le \dbE\bigl[X_T(T)^\top P_\infty X_T(T)\bigr]
		\le \cC e^{-2\gamma(T-t)}|x|^2.
		%&&\text{(by \eqref{terminal-decay})}.
	\end{aligned}
\end{equation*}
Both estimates use only $|z^\top P_\infty z|\le\|P_\infty\||z|^2$ and do not use any nonnegativity property of $P_T(t)$ or $P_\infty$.
Thus,
\[
    |x^\top(P_T(t)-P_\infty)x| \le \cC e^{-2\gamma(T-t)}|x|^2.
\]
Since $P_T(t)-P_\infty$ is a real symmetric matrix, taking the supremum over $|x|=1$ and relabeling the exponent gives \eqref{Riccati-convergence}.
The same estimate also yields the uniform boundedness of $P_T(t)$.
\end{proof}

\begin{proposition}\label{prop:riccati-convergence}
	%Suppose that \ref{ass:L2-stabilizable} and \ref{ass:JC} hold.
		%Then there exist constants $\cC,\gamma>0$, independent of $T$ and $t$, such that
	Suppose that \ref{ass:L2-stabilizable} and \ref{ass:JC} hold.
		Let $P_T$ be the unique strongly regular solution of \eqref{Riccati-T},
		and let $P_\infty$ be the unique stabilizing solution of \eqref{ARE}.
		Then there exist constants $\cC,\gamma>0$, independent of $T$ and $t$,
		such that
	%Then there exist constants $\cC,\gamma>0$, independent of $T$ and $t$, such that
	\begin{equation}\label{JC-convergence}
		\|P_T(t)-P_\infty\|\le\cC e^{-\gamma(T-t)},\qquad 0\le t\le T.
	\end{equation}
\end{proposition}

\begin{proof}
	By stabilizability, there exists a stabilizing feedback matrix $\Theta$.
	Set $A_\Theta=A+B\Theta$ and $C_\Theta=C+D\Theta$.
	We denote the	transformed finite-horizon cost by
	$J_{\Theta,T}^0(x;v)$, its value matrix by	$P_{\Theta,T}(t)$, and the infinite-horizon	value matrix by $P_{\Theta,\infty}$.
	Under the change of control $u(t)=\Theta X(t)+v(t)$, the state equation~\eqref{state-finite} and the cost coefficients become
	\begin{equation}\label{state-Theta-finite}
		\left\{
		\begin{aligned}
			&dX(t)=\bigl(A_\Theta X(t)+Bv(t)\bigr)\,dt
			+\bigl(C_\Theta X(t)+Dv(t)\bigr)\,dW(t),
			\qquad 0\le t\le T,\\
			&X(0)=x.
		\end{aligned}
		\right.
	\end{equation}
	\begin{equation}\label{transformed-cost-coefficients}
		\begin{aligned}
			Q_\Theta=Q+S^\top\Theta+\Theta^\top S+\Theta^\top R\Theta,\qquad
			S_\Theta=S+R\Theta,\qquad R_\Theta=R.
		\end{aligned}
	\end{equation}
	For each fixed $T$, the transformations $u(t)=\Theta X(t)+v(t)$ and $v(t)=u(t)-\Theta X(t)$ are mutually inverse and preserve the initial state, the state trajectory, and the cost.
	The assumption \ref{ass:JC}, together with
$	|v(t)|^2\le(1+\|\Theta\|^2)(|X(t)|^2+|u(t)|^2),$	yields
	
		\begin{equation}\label{transformed-UC-finite}
		J_{\Theta,T}^0(0;v)\ge
		\frac{\delta}{1+\|\Theta\|^2}\dbE\int_0^T|v(t)|^2dt.
	\end{equation}

	On the infinite horizon, by Sun--Yong~\cite[Proposition~3.3]{SunYong2018}, the transformation $u(t)=\Theta X(t)+v(t)$ defines a bijection from
	$L^2_{\dbF}(0,\infty;\dbR^m)$ onto $\mathcal U_{\mathrm{ad}}(x)$, with inverse $v(t)=u(t)-\Theta X(t)$.
	Thus, on both finite and infinite horizons, corresponding controls yield the same state trajectory and the same cost.
	The value functions therefore coincide.
	Their quadratic representations give $P_{\Theta,T}(0)=P_T(0)$ and $P_{\Theta,\infty}=P_\infty$.
	Moreover, time homogeneity yields $P_{\Theta,T}(t)=P_{\Theta,T-t}(0)=P_{T-t}(0)=P_T(t)$.
	The value matrices before and after the transformation satisfy
	\begin{equation*}
		P_{\Theta,T}(t)=P_T(t),\qquad P_{\Theta,\infty}=P_\infty,
		\qquad 0\le t\le T.
	\end{equation*}
Since $[A_\Theta,C_\Theta]$ is $L^2$-stable and
\eqref{transformed-UC-finite} holds with a constant
independent of $T$, Lemma~\ref{lem:riccati-convergence}
applies to \eqref{state-Theta-finite} with the cost
coefficients \eqref{transformed-cost-coefficients}.
Using $P_{\Theta,T}(t)=P_T(t)$ and
$P_{\Theta,\infty}=P_\infty$, we obtain
\eqref{JC-convergence}.
\end{proof}

The above implies that the $JC$ condition \ref{ass:JC} implies the
convergence estimate \rf{JC-convergence} of Riccati equations.
This convergence estimate will play the most important role
in the study of the turnpike property.
Next, we prove the converse result; that is the  convergence of Riccati equations implies  the $JC$ condition \ref{ass:JC} holds.

\begin{theorem}\label{Thm-equiv}
	Consider the finite-horizon Riccati equation \eqref{Riccati-T}, with the notation $M_T,N_T,K_T$ in Proposition~\ref{prop:exploratory-solution}.
	Suppose that the following conditions hold:
	\begin{enumerate}[(i)]
		\item For every $T>0$, the  Riccati equation \eqref{Riccati-T} admits a strongly regular solution
		$P_T\in C([0,T];\dbS^n)$. More precisely, there exists $\mu_T>0$, such that
		\[
		M_T(t)=R+D^\top P_T(t)D\ge\mu_T I_m,
		\qquad 0\le t\le T.
		\]
		\item There exists $P_\infty\in\dbS^n$ such that
		\[
		\lim_{T\to\infty}\|P_T(0)-P_\infty\|=0,
		\]
		and  $[A-BK_\infty,C-DK_\infty]$ is $L^2$-stable, where
      $$ M_\infty:=R+D^\top P_\infty D>0,\qq
		K_\infty=M_\infty^{-1}
		(B^\top P_\infty+D^\top P_\infty C+S).$$
	\end{enumerate}
	Then $P_\infty$ is the unique stabilizing solution of the algebraic Riccati equation \eqref{ARE}, and there exists $\delta>0$, independent of $T$, such that
	\[
	J_T^0(0;u)\ge\delta\dbE\int_0^T
	\bigl(|X(t)|^2+|u(t)|^2\bigr)dt,
	\qquad T>0,\quad u\in L^2_{\dbF}(0,T;\dbR^m).
	\]
	That is, \ref{ass:JC} holds.

\end{theorem}

\begin{proof}

	We first establish a uniform positive lower bound for $M_T(t)$.
	Since the coefficients are constant and $P_T(T)=0$ for every $T>0$, time reversal transforms the finite-horizon Riccati equations into the same initial-value problem; see \cite{SunWangYong2022,SunYong2018}.
	By uniqueness, there exists a function $\Pi\in C^1([0,\infty);\dbS^n)$ such that
	\[
	P_T(t)=\Pi(T-t),\qquad 0\le t\le T,\qquad \Pi(0)=0.
	\]
	In particular, $\Pi(T)=P_T(0)$, so the assumed convergence implies that $\Pi(s)\to P_\infty$ as $s\to\infty$.
	Define
	\[
	M(s)=R+D^\top\Pi(s)D,\qquad N(s)=B^\top\Pi(s)+D^\top\Pi(s)C+S,\qquad K(s)=M(s)^{-1}N(s).
	\]
	Then $M_T(t)=M(T-t)$ and $K_T(t)=K(T-t)$ for $0\le t\le T$.
	From the Riccati equation \eqref{Riccati-T}, we obtain the following time-reversed equation:
	\begin{equation}\label{riccati-reversed}
		\left\{
		\begin{aligned}
			&\Pi'(s)=\Pi(s)A+A^\top\Pi(s)+C^\top\Pi(s)C+Q
			-N(s)^\top M(s)^{-1}N(s),
			\qquad s\ge0,\\
			&\Pi(0)=0.
		\end{aligned}
		\right.
	\end{equation}
	Strong regularity on each finite horizon gives $M(s)>0$ for all $s\ge0$.
	Moreover, $M(s)\to M_\infty>0$ as $s\to\infty$, so there exists $s_0>0$ such that
	\[
	M(s)\ge\tfrac12\lambda_{\min}(M_\infty)I_m,
	\qquad s\ge s_0.
	\]
	Since $M(s)$ is continuous and positive definite on the compact interval $[0,s_0]$, we also have
	$\mu_0:=\min_{0\le s\le s_0}\lambda_{\min}(M(s))>0$.
	Combining these bounds and setting
	$\mu=\min\{\mu_0,\tfrac12\lambda_{\min}(M_\infty)\}>0$,
	we obtain
	\[
	M(s)\ge\mu I_m,\qquad
	\|M(s)^{-1}\|\le\mu^{-1},
	\qquad s\ge0.
	\]
	
	Since $\Pi(s)$ is continuous and convergent, $N(s)$ is bounded on $[0,\infty)$.
	The uniform bound on $M(s)^{-1}$ then implies that $K(s)$ is also bounded.
	Moreover, the definitions of $K(s)$ and $K_\infty$ give
	\begin{equation*}
		\begin{aligned}
			K(s)-K_\infty
			% &=M(s)^{-1}\bigl[M(s)K(s)-M(s)K_\infty\bigr]\\
			% &=M(s)^{-1}\Bigl[M(s)K(s)-M_\infty K_\infty
			% -(M(s)-M_\infty)K_\infty\Bigr]\\
			% &=M(s)^{-1}\Bigl[B^\top\Pi(s)-B^\top P_\infty
			% +D^\top\Pi(s)C-D^\top P_\infty C\\
			% &\qquad\qquad\quad
			% -\bigl(D^\top\Pi(s)D-D^\top P_\infty D\bigr)K_\infty\Bigr]\\
			% &=M(s)^{-1}\Bigl[B^\top(\Pi(s)-P_\infty)
			% +D^\top(\Pi(s)-P_\infty)C\\
			% &\qquad\qquad\quad
			% -D^\top(\Pi(s)-P_\infty)DK_\infty\Bigr]\\
			&=M(s)^{-1}\Bigl[B^\top(\Pi(s)-P_\infty)
			+D^\top(\Pi(s)-P_\infty)(C-DK_\infty)\Bigr].
		\end{aligned}
	\end{equation*}
	Using $\|M(s)^{-1}\|\le\mu^{-1}$, we obtain
	\begin{equation*}
		\|K(s)-K_\infty\|
		\le\cC\|\Pi(s)-P_\infty\|\longrightarrow0,
		\qquad s\to\infty,
	\end{equation*}
	where $\cC>0$ is independent of $s$.
	
	We next show that $P_\infty$ satisfies \eqref{ARE}.
	Since $\Pi(s)\to P_\infty$ and $M(s)\to M_\infty>0$, every term on the right-hand side of \eqref{riccati-reversed} converges.
	Hence $\Pi'(s)$ has a limit $L\in\dbS^n$.
	The identity
	\begin{equation*}
		\Pi(s+1)-\Pi(s)=\int_s^{s+1}\Pi'(r)\,dr
	\end{equation*}
	then implies $L=0$, since the left-hand side tends to zero and the right-hand side tends to $L$ as $s\to\infty$.
	Thus $\Pi'(s)\to0$, and passing to the limit in \eqref{riccati-reversed} yields \eqref{ARE} for $P_\infty$.

	Fix $T>0$ and $u\in L^2_{\dbF}(0,T;\dbR^m)$, and let $X$ be the corresponding solution of \eqref{state-homo} with $x=0$, namely,
	\begin{equation}\label{Riccati-zero-state}
		\left\{
		\begin{aligned}
			&dX(t)=(AX(t)+Bu(t))\,dt+(CX(t)+Du(t))\,dW(t),
			\qquad 0\le t\le T,\\
			&X(0)=0.
		\end{aligned}
		\right.
	\end{equation}
	Set
	\begin{equation}\label{Riccati-v}
		v(t)=u(t)+K_T(t)X(t),\qquad 0\le t\le T.
	\end{equation}
	The standard finite-horizon estimate gives
	$X\in L^2_{\dbF}(0,T;\dbR^n)$.
	Since $K_T$ is bounded on $[0,T]$, we also have
	$v\in L^2_{\dbF}(0,T;\dbR^m)$.
	Applying It\^o's formula to $X(t)^\top P_T(t)X(t)$,
	taking expectations, and using \eqref{Riccati-T}, we obtain
	\begin{equation*}
		\begin{aligned}
			\frac{d}{dt}\dbE[X(t)^\top P_T(t)X(t)]
			&=\dbE\bigl[
			X(t)^\top\bigl(\dot P_T(t)+A^\top P_T(t)
			+P_T(t)A+C^\top P_T(t)C\bigr)X(t)\\
			&\qquad
			+2u(t)^\top\bigl(B^\top P_T(t)+D^\top P_T(t)C\bigr)X(t)
			+u(t)^\top D^\top P_T(t)Du(t)
			\bigr]\\
			&=\dbE\bigl[
			X(t)^\top\bigl(K_T(t)^\top M_T(t)K_T(t)-Q\bigr)X(t)\\
			&\qquad
			+2u(t)^\top\bigl(M_T(t)K_T(t)-S\bigr)X(t)
			+u(t)^\top\bigl(M_T(t)-R\bigr)u(t)
			\bigr]\\
			% &={}\dbE\bigl[
			% (K_T(t)X(t))^\top M_T(t)(K_T(t)X(t))\\
			% &\qquad
			% +2u(t)^\top M_T(t)(K_T(t)X(t))
			% +u(t)^\top M_T(t)u(t)\\
			% &\qquad
			% -X(t)^\top QX(t)
			% -2u(t)^\top SX(t)-u(t)^\top Ru(t)
			% \bigr]\\
			&=\dbE\bigl[
			(u(t)+K_T(t)X(t))^\top M_T(t)(u(t)+K_T(t)X(t))\\
			&\qquad
			-X(t)^\top QX(t)-2u(t)^\top SX(t)-u(t)^\top Ru(t)
			\bigr]\\
			&=\dbE\bigl[
			v(t)^\top M_T(t)v(t)-X(t)^\top QX(t)
			-2u(t)^\top SX(t)-u(t)^\top Ru(t)
			\bigr].
		\end{aligned}
	\end{equation*}
	Integrating over $[0,T]$ gives
	\begin{equation*}
		\begin{aligned}
			&\dbE[X(T)^\top P_T(T)X(T)]-X(0)^\top P_T(0)X(0)\\
			&=\dbE\int_0^T\bigl[
			v(t)^\top M_T(t)v(t)-X(t)^\top QX(t)
			-2u(t)^\top SX(t)-u(t)^\top Ru(t)
			\bigr]\,dt.
		\end{aligned}
	\end{equation*}
	Since $X(0)=0$, $P_T(T)=0$, and $M_T(t)\ge\mu I_m$, it follows that
	\begin{equation}\label{Riccati-square}
		J_T^0(0;u)
		=\dbE\int_0^T v(t)^\top M_T(t)v(t)\,dt
		\ge\mu\dbE\int_0^T|v(t)|^2\,dt.
	\end{equation}
	
	Let $h>0$ be a constant to be specified later, independently of $T$,
	and set $t_0=\max\{T-h,0\}$.
	We first estimate the state on the terminal interval $[t_0,T]$,
	whose length is at most $h$.
	Substituting $u(t)=v(t)-K_T(t)X(t)$ into
	\eqref{Riccati-zero-state} and setting $\xi=X(t_0)$,
	we see that the same state process $X$ satisfies
	\begin{equation}\label{Riccati-terminal-state}
		\left\{
		\begin{aligned}
			&dX(t)={}[(A-BK_T(t))X(t)+Bv(t)]\,dt
			+[(C-DK_T(t))X(t)+Dv(t)]\,dW(t),
			\qquad t_0\le t\le T,\\
			&X(t_0)={}\xi.
		\end{aligned}
		\right.
	\end{equation}
	Applying It\^o's formula to $|X(t)|^2$ along the solution of
	\eqref{Riccati-terminal-state} and taking expectations,
	we obtain, for almost every $t\in[t_0,T]$,
	\begin{equation*}
		\begin{aligned}
			\frac{d}{dt}\dbE|X(t)|^2
			&=\dbE\Bigl[
			X(t)^\top\Bigl(
			(A-BK_T(t))^\top+(A-BK_T(t))
			+(C-DK_T(t))^\top(C-DK_T(t))
			\Bigr)X(t)\\
			&\qquad
			+v(t)^\top D^\top Dv(t)
			+2v(t)^\top\bigl(B^\top+D^\top(C-DK_T(t))\bigr)X(t)
			\Bigr]\\
			&\le a\dbE|X(t)|^2+b\dbE|v(t)|^2,
		\end{aligned}
	\end{equation*}
	where the last inequality follows from Young's inequality
	and the uniform boundedness of $K_T(t)=K(T-t)$,
	and the constants $a>0$ and $b\ge0$ are independent of $T$, $u$, and $h$.
	Gronwall's inequality gives, for $t_0\le t\le T$,
	\begin{equation*}
		\begin{aligned}
			\dbE|X(t)|^2
			&\le e^{a(t-t_0)}\dbE|X(t_0)|^2
			+b\dbE\int_{t_0}^t e^{a(t-s)}|v(s)|^2\,ds\\
			&\le e^{ah}\left[
			\dbE|X(t_0)|^2
			+b\dbE\int_{t_0}^T|v(s)|^2\,ds
			\right].
		\end{aligned}
	\end{equation*}
	Integrating over $[t_0,T]$ and using $T-t_0\le h$, we obtain
	\begin{equation}\label{Riccati-terminal-estimate}
		\dbE\int_{t_0}^T|X(t)|^2\,dt
		\le he^{ah}\left[
		\dbE|X(t_0)|^2
		+b\dbE\int_{t_0}^T|v(t)|^2\,dt
		\right].
	\end{equation}

	We next estimate the state on $[0,t_0]$.
	Set
	\[
	A_\infty=A-BK_\infty,\qquad C_\infty=C-DK_\infty,
	\qquad w(t)=u(t)+K_\infty X(t).
	\]
	Then the same state process $X$ as in \eqref{Riccati-zero-state}
	satisfies
	\begin{equation}\label{Riccati-input-state}
		\left\{
		\begin{aligned}
			&dX(t)=(A_\infty X(t)+Bw(t))\,dt
			+(C_\infty X(t)+Dw(t))\,dW(t),
			\qquad 0\le t\le T,\\
			&X(0)=0.
		\end{aligned}
		\right.
	\end{equation}
	Since $[A_\infty,C_\infty]$ is $L^2$-stable and
	$w\in L^2_{\dbF}(0,T;\dbR^m)$,
	Lemma~\ref{lem:state-estimate}, applied to
	\eqref{Riccati-input-state}, gives a constant $\cC>0$,
	independent of $T$, $u$, and $h$, such that
	\begin{equation}\label{Riccati-input-estimate}
		\dbE\int_0^s|X(t)|^2\,dt+\dbE|X(s)|^2
		\le\cC\dbE\int_0^s|w(t)|^2\,dt,
		\qquad 0\le s\le T.
	\end{equation}
	
	For $h\ge0$, define
	\[
	\varepsilon(h)=\sup_{s\ge h}\|K(s)-K_\infty\|.
	\]
	Since $K(s)\to K_\infty$ as $s\to\infty$,
	we have $\varepsilon(h)\to0$ as $h\to\infty$.
	Moreover, $K_T(t)=K(T-t)$ implies that
	\[
	\|K_T(t)-K_\infty\|\le\varepsilon(h),
	\qquad 0\le t\le T-h,\quad T>h.
	\]
	Choose $h>0$, independently of $T$ and $u$, sufficiently large that
	$2\cC\varepsilon(h)^2\le1/2$.
	If $T>h$, then $t_0=T-h$.
	By \eqref{Riccati-v}, we have
	\[
	w(t)=v(t)+(K_\infty-K_T(t))X(t),
	\qquad 0\le t\le t_0,
	\]
	and hence
	\[
	|w(t)|^2
	\le2|v(t)|^2+2\varepsilon(h)^2|X(t)|^2,
	\qquad 0\le t\le t_0.
	\]
	Taking $s=t_0$ in \eqref{Riccati-input-estimate}, we obtain
	\[
	\begin{aligned}
		\dbE\int_0^{t_0}|X(t)|^2\,dt+\dbE|X(t_0)|^2
		&\le2\cC\dbE\int_0^{t_0}|v(t)|^2\,dt
		+2\cC\varepsilon(h)^2\dbE\int_0^{t_0}|X(t)|^2\,dt\\
		&\le2\cC\dbE\int_0^{t_0}|v(t)|^2\,dt
		+\frac12\dbE\int_0^{t_0}|X(t)|^2\,dt.
	\end{aligned}
	\]
	Absorbing the last term into the left-hand side yields
	\begin{equation}\label{Riccati-front-estimate}
		\begin{aligned}
			\dbE\int_0^{t_0}|X(t)|^2\,dt
			\le4\cC\dbE\int_0^{t_0}|v(t)|^2\,dt,\qq
			\dbE|X(t_0)|^2
			\le2\cC\dbE\int_0^{t_0}|v(t)|^2\,dt.
		\end{aligned}
	\end{equation}
	If $T\le h$, then $t_0=0$, and
	\eqref{Riccati-front-estimate} remains valid because $X(0)=0$.

	Combining \eqref{Riccati-terminal-estimate} with
	\eqref{Riccati-front-estimate}, we obtain
	\begin{equation}\label{Riccati-energy-estimate}
		\dbE\int_0^T|X(t)|^2\,dt
		\le\cC\dbE\int_0^T|v(t)|^2\,dt,
	\end{equation}
	where $\cC>0$ is independent of $T$ and $u$ and is chosen
	large enough that $\|K_T(t)\|\le\cC$ for all $T>0$ and $0\le t\le T$.
	Since $u(t)=v(t)-K_T(t)X(t)$ by \eqref{Riccati-v}, we have
	\[
	|u(t)|^2
	\le2|v(t)|^2+2\cC^2|X(t)|^2,
	\qquad 0\le t\le T.
	\]
	Combining \eqref{Riccati-energy-estimate} and
	\eqref{Riccati-square}, we obtain
	\begin{equation*}
		\begin{aligned}
			\dbE\int_0^T\bigl(|X(t)|^2+|u(t)|^2\bigr)\,dt
			&\le\bigl[2+(1+2\cC^2)\cC\bigr]
			\dbE\int_0^T|v(t)|^2\,dt
			\le\frac{2+(1+2\cC^2)\cC}{\mu}\,J_T^0(0;u).
		\end{aligned}
	\end{equation*}
	Thus, taking $\delta=\mu/[2+(1+2\cC^2)\cC]>0$,
	we obtain \ref{ass:JC}.
\end{proof}

\begin{remark}
%The $JC$ condition \ref{ass:JC} implies that
%\begin{equation}\label{JC-infinite-11}
%    J_\infty^0(0;u)\geq\delta\dbE\int_0^\infty
%    \bigl(|X^{0,u}(t)|^2+|u(t)|^2\bigr)\,dt,
%    \qquad \forall u\in\mathcal U_{\mathrm{ad}}(0),
%\end{equation}
%which is precisely \rf{JC-infinite}, and that, for every \(T>0\), there exists a constant \(\delta_T>0\) such that
Assumption~\ref{ass:JC} implies \eqref{JC-infinite} and, for every
	$T>0$, a finite-horizon lower bound with a constant $\delta_T>0$:
\begin{equation}\label{UC-T1}
    J_T^0(0;u)\ge\delta_T\dbE\int_0^T\big(|X(t)|^2+|u(t)|^2\big)\,dt,
    \qquad T>0,\quad u\in L^2_{\dbF}(0,T;\dbR^m).
\end{equation}
Under \eqref{JC-infinite} and \eqref{UC-T1}, the Riccati equations associated with both the finite-horizon and infinite-horizon problems are solvable. However, \eqref{JC-infinite} and \eqref{UC-T1} do not imply that \ref{ass:JC} holds. Consequently, the corresponding Riccati equations need not converge to the stabilizing solution. By Theorem \ref{Thm-equiv}, Example \ref{example1} is readily seen to provide a counterexample.
%Under \rf{JC-infinite-11}--\rf{UC-T1}, the Riccati equations associated with both the finite-horizon and infinite-horizon %problems are solvable. However, \rf{JC-infinite-11}--\rf{UC-T1} do not imply that \ref{ass:JC} holds. Consequently, the %corresponding Riccati equations need not converge to the stabilizing solution. By Theorem \ref{Thm-equiv}, Example %\ref{example1} is readily seen to provide a counterexample.
\end{remark}

\begin{remark}
By Remark~\ref{remark2.10}, \eqref{UC-T1} and \eqref{UC} are
equivalent on each fixed horizon.
As shown in Sun--Li--Yong \cite{Sun-Li-Yong2016}, \rf{UC} is almost necessary for the solvability of the Riccati equation \rf{Riccati-T}. We establish the convergence result in Proposition \ref{prop:riccati-convergence} under the $JC$ condition \ref{ass:JC}, which is slightly stronger than the uniform convexity condition \rf{UC}. Furthermore, by Theorem \ref{Thm-equiv}, the $JC$ condition \ref{ass:JC} is necessary for the convergence of Riccati equations. Thus, the $JC$ condition \ref{ass:JC} is sharp.
\end{remark}

For the remainder of this section, suppose that
Assumptions~\ref{ass:L2-stabilizable} and \ref{ass:JC} hold.
%Let
%\begin{equation*}
%\begin{aligned}
%& M_T(t)=R+D^\top P_T(t)D, \qquad N_T(t)=B^\top P_T(t)+D^\top P_T(t)C+S,\\
%& M_\infty=R+D^\top P_\infty D, \qquad N_\infty=B^\top P_\infty+D^\top P_\infty C+S.
%\end{aligned}
%\end{equation*}
%For every $T>0$, Proposition~\ref{prop:strong-regularity} gives $M_T(t)\ge\delta I_m$ for all $0\le t\le T$.
%Taking $t=0$, letting $T\to\infty$, and using \eqref{JC-convergence} in Proposition~\ref{prop:riccati-convergence}, we obtain %$M_\infty\ge\delta I_m$.
%Therefore, $M_T(t)\ge\delta I_m$ for $0\le t\le T$, and $M_\infty\ge\delta I_m$.
%The finite-horizon and infinite-horizon feedback gains are defined by
%\begin{equation*}
%    K_T(t)=M_T(t)^{-1}N_T(t), \qquad K_\infty=M_\infty^{-1}N_\infty.
%\end{equation*}
Recall the finite-horizon coefficients $M_T,N_T,K_T$ from
	Proposition~\ref{prop:exploratory-solution} and the limiting
	coefficients $M_\infty,N_\infty,K_\infty$ from \eqref{gain}.
	For every $T>0$, Proposition~\ref{prop:strong-regularity} gives
	$M_T(t)\ge\delta I_m$ for all $0\le t\le T$.
	Taking $t=0$, letting $T\to\infty$, and using
	\eqref{JC-convergence} in Proposition~\ref{prop:riccati-convergence},
	we obtain $M_\infty\ge\delta I_m$.
The following corollary follows  from standard matrix-norm estimates.

\begin{corollary}
There exist constants $\cC,\gamma>0$, independent of $T$, such that
\begin{equation}
    \label{coefficient-decay}
    \begin{aligned}
    & \|M_T(t)-M_\infty\|+\|N_T(t)-N_\infty\| \le \cC e^{-\gamma(T-t)}, \qquad 0\le t\le T,\\
    & \|M_T(t)^{-1}-M_\infty^{-1}\| \le \cC e^{-\gamma(T-t)}, \qquad 0\le t\le T.
   \end{aligned}
\end{equation}
\end{corollary}

\begin{corollary}
There exist constants $\cC,\gamma>0$, independent of $T$, such that
\begin{equation*}
    \|K_T(t)-K_\infty\| \le \cC e^{-\gamma(T-t)}, \qquad 0\le t\le T.
\end{equation*}
\end{corollary}

%By Proposition~\ref{prop:exploratory-solution}, the linear term $\varphi_T$ in the finite-horizon HJB equation satisfies
%\begin{equation}
 %   \left\{
 %   \begin{aligned}
 %   &\dot\varphi_T(t) +\big(A-BK_T(t)\big)^\top\varphi_T(t) +\big(C-DK_T(t)\big)^\top P_T(t)\sigma\\
%  &\quad-K_T(t)^\top r+P_T(t)b+q=0, \qquad t\in[0,T],\\
%    &\varphi_T(T)=0.
%    \end{aligned}
 %   \right.
%\end{equation}
Recall that $\varphi_T$ solves the terminal-value problem
	\eqref{linear-terminal-equation}.
Since $P_T(t)$ and $K_T(t)$ will converge to $P_\infty$ and $K_\infty$ as $T-t\to\infty$, we consider the following algebraic equation:
\begin{equation}
    \label{affine-equation}
    \big(A-BK_\infty\big)^\top\varphi_\infty +\big(C-DK_\infty\big)^\top P_\infty\sigma -K_\infty^\top r+P_\infty b+q=0 .
\end{equation}
By Corollary~\ref{cor:JC-Damm}, the closed-loop system $[A-BK_\infty,C-DK_\infty]$ is $L^2$-stable.
Corollary~\ref{cor:Hurwitz} therefore implies that $A-BK_\infty$ is Hurwitz and hence invertible.
Thus \eqref{affine-equation} has a unique solution, which we denote by $\varphi_\infty$.

\begin{lemma}
Let $\varphi_T$ and $\varphi_\infty$ be the unique solutions of \eqref{linear-terminal-equation} and \eqref{affine-equation}, respectively.
	There exist constants $\cC,\gamma>0$, independent of $T$, such that
\begin{equation*}
    |\varphi_T(t)-\varphi_\infty| \le \cC e^{-\gamma(T-t)}, \qquad 0\le t\le T.
\end{equation*}
\end{lemma}

\begin{proof}
The proof is the same as that in Sun--Wang--Yong \cite[Lemma~5.1]{SunWangYong2022}.
\end{proof}

Define
%\begin{align*}
 %  & d_T(t)=B^\top\varphi_T(t)+D^\top P_T(t)\sigma+r, \q d_\infty=B^\top\varphi_\infty+D^\top P_\infty\sigma+r,\q %\ell_T(t)=M_T(t)^{-1}d_T(t),\\
%   & \ell_\infty=M_\infty^{-1}d_\infty,\q \bar u_T(t,x)=-K_T(t)x-\ell_T(t), \q \bar u_\infty(x)=-K_\infty x-\ell_\infty.
%\end{align*}
%

	\begin{equation*}
		\begin{alignedat}{2}
			&d_T(t)=B^\top\varphi_T(t)+D^\top P_T(t)\sigma+r,
			\qquad
			&&d_\infty=B^\top\varphi_\infty+D^\top P_\infty\sigma+r,\\
			&\ell_T(t)=M_T(t)^{-1}d_T(t),
			\qquad
			&&\ell_\infty=M_\infty^{-1}d_\infty,\\
			&\bar u_T(t,x)=-K_T(t)x-\ell_T(t),
			\qquad
			&&\bar u_\infty(x)=-K_\infty x-\ell_\infty.
		\end{alignedat}
	\end{equation*}

With the above estimates, we can  obtain the following corollary.

\begin{corollary}
There exist constants $\cC,\gamma>0$, independent of $T$, such that
\begin{equation*}
    \|K_T(t)-K_\infty\|+|\ell_T(t)-\ell_\infty| \le \cC e^{-\gamma(T-t)}, \qquad 0\le t\le T.
\end{equation*}
Thus, for any $x\in\dbR^n$,
\begin{equation*}
    |\bar u_T(t,x)-\bar u_\infty(x)| \le \cC(1+|x|)e^{-\gamma(T-t)}.
\end{equation*}
\end{corollary}

%The finite-horizon exploration covariance and the stationary exploration covariance are
%\begin{equation}\label{covariance-recall}
%    \Lambda_T^\lambda(t)=\frac{\lambda}{2}M_T(t)^{-1}, \qquad \Lambda_\infty^\lambda=\frac{\lambda}{2}M_\infty^{-1}.
%\end{equation}

Recall the exploration covariances from
	Proposition~\ref{prop:exploratory-solution} and
	\eqref{optimal-covariance}.
\begin{lemma}
\label{lem:covariance-convergence}
There exist constants $\cC,\gamma>0$, independent of $T$ and $\lambda$, such that
\begin{align}
    \label{covariance-decay}
    &\|\Lambda_T^\lambda(t)-\Lambda_\infty^\lambda\| \le \cC\lambda e^{-\gamma(T-t)}, \qquad 0\le t\le T,\\
    \label{sqrt-convergence}
    &\|\big(\Lambda_T^\lambda(t)\big)^{1/2}-\big(\Lambda_\infty^\lambda\big)^{1/2}\| \le \cC\sqrt\lambda e^{-\gamma(T-t)}, \qquad 0\le t\le T.
\end{align}
\end{lemma}

\begin{proof}
%By \eqref{covariance-recall}, the estimate $M_T(t),M_\infty\ge\delta I_m$, and \eqref{coefficient-decay}, we have
Using the covariance formulas in
	Proposition~\ref{prop:exploratory-solution} and
	\eqref{optimal-covariance}, together with the bounds
	$M_T(t),M_\infty\ge\delta I_m$ and \eqref{coefficient-decay},
	we have
\begin{equation}
    \label{covariance-estimate}
    \begin{aligned}
    \|\Lambda_T^\lambda(t)-\Lambda_\infty^\lambda\| &=\frac{\lambda}{2}\|M_T(t)^{-1}-M_\infty^{-1}\|=\frac{\lambda}{2} \|M_T(t)^{-1}(M_\infty-M_T(t))M_\infty^{-1}\|\\
    &\le\frac{\lambda}{2\delta^2}\|M_T(t)-M_\infty\|\le \cC\lambda e^{-\gamma(T-t)}.
    \end{aligned}
\end{equation}
This proves \eqref{covariance-decay}.
By the uniform boundedness of $P_T(t)$ in Proposition~\ref{prop:riccati-convergence} and the definitions of $M_T(t)$ and $M_\infty$, there exists a constant $\beta>0$, independent of $T,t,\lambda$, such that
\begin{equation*}
    \delta I_m\le M_T(t)\le\beta I_m, \qquad \delta I_m\le M_\infty\le\beta I_m.
\end{equation*}
Hence
\begin{equation}
    \label{spectral-bound}
    \begin{aligned}
    M_T(t)^{-1}\ge\beta^{-1}I_m,\q M_\infty^{-1}&\ge\beta^{-1}I_m,\q
    \Lambda_T^\lambda(t)\ge\frac{\lambda}{2\beta}I_m, \q\Lambda_\infty^\lambda&\ge\frac{\lambda}{2\beta}I_m.
    \end{aligned}
\end{equation}
Note from Higham \cite[Theorem~6.2]{Higham2008} that
\[
    \|\Lambda^{1/2}-\widehat\Lambda^{1/2}\| \le \frac{\|\Lambda-\widehat\Lambda\|} {\sqrt{\lambda_{\min}(\Lambda)} +\sqrt{\lambda_{\min}(\widehat\Lambda)}},
\]
 for any $\Lambda,\widehat\Lambda\in\dbS_{++}^m$,  where $\lambda_{\min}(\cdot)$ denotes the smallest eigenvalue of the corresponding matrix.
%Substituting $\Lambda=\Lambda_T^\lambda(t)$ and $\widehat\Lambda=\Lambda_\infty^\lambda$ into the above, and combining %\eqref{covariance-estimate} and \eqref{spectral-bound}, we obtain
%\begin{equation*}
%    \|\big(\Lambda_T^\lambda(t)\big)^{1/2} -\big(\Lambda_\infty^\lambda\big)^{1/2}\| \le \cC\sqrt\lambda e^{-\gamma(T-t)}.
%\end{equation*}
%All the constants above are independent of $T,t,\lambda$, which proves \eqref{sqrt-convergence}.
Combining this inequality with \eqref{covariance-estimate} and
	\eqref{spectral-bound} gives \eqref{sqrt-convergence}, with constants
	independent of $T,t,\lambda$.
\end{proof}

Combining the preceding results, we conclude the following exponential convergence estimate between the finite-horizon and infinite-horizon parameters.

\begin{proposition}
\label{prop:parameter-convergence}
Let \ref{ass:L2-stabilizable}--\ref{ass:JC} hold.
Then there exist constants $\cC,\gamma>0$, independent of $T$, such that for all $0\le t\le T$,
\begin{equation}
    \begin{aligned}
    &\|P_T(t)-P_\infty\| +|\varphi_T(t)-\varphi_\infty| +\|K_T(t)-K_\infty\| +|\ell_T(t)-\ell_\infty|+\lambda^{-1}\|\Lambda_T^\lambda(t)-\Lambda_\infty^\lambda\|\\
    & +\lambda^{-1/2} \|\big(\Lambda_T^\lambda(t)\big)^{1/2}-\big(\Lambda_\infty^\lambda\big)^{1/2}\|\le \cC e^{-\gamma(T-t)}.
    \end{aligned}
\end{equation}
Therefore, for any $x\in\dbR^n$,
\begin{equation*}
    |\bar u_T(t,x)-\bar u_\infty(x)| \le \cC(1+|x|)e^{-\gamma(T-t)}.
\end{equation*}
\end{proposition}

\section{Turnpike Properties}
\label{sec:turnpike-properties}

Using the Gaussian feedback and parameter convergence established above, we construct the stochastic reference process $X^{*,\lambda}$ from the infinite-horizon closed-loop system starting at $x_*$.
We first prove the state-process turnpike estimate and then derive the Wasserstein estimate for the associated control distributions.
The reference object is a process because diffusion and exploration noise generally persist even when the state starts at the static center.

Using the long-time feedback parameters from Section~\ref{sec:riccati-parameter-convergence}, define the infinite-horizon Gaussian feedback distribution
\begin{equation}
    \label{Gaussian-feedback}
    \pi_\infty^{\lambda,*}(\cdot\mid x) = \cNcal\big(-K_\infty x-\ell_\infty,\Lambda_\infty^\lambda\big).
\end{equation}
Before constructing the reference process, we first show that this long-time feedback is compatible, at the static center, with the static distribution optimizer from Section~\ref{sec:static-distribution-optimization}.

\begin{proposition}
\label{prop:static-feedback-compatibility}
Let \ref{ass:L2-stabilizable}--\ref{ass:JC} hold.
Let $(x_*,u_*)$ be the unique optimizer of the static problem \eqref{static-problem}, and let $\zeta^*$ be a Lagrange multiplier satisfying the KKT system \eqref{KKT-system}.
%That is,
%\begin{equation*}
%    \left\{
%    &Ax_*+Bu_*+b=0,\\
%    &Qx_*+S^\top u_*+q+C^\top P_\infty(Cx_*+Du_*+\sigma)+A^\top \zeta^*=0,\\
%    &Sx_*+Ru_*+r+D^\top P_\infty(Cx_*+Du_*+\sigma)+B^\top \zeta^*=0.
%    \end{aligned}
%    \right.
%\end{equation*}
Then
\begin{equation}
    \label{zeta-feedback}
    \zeta^*=P_\infty x_*+\varphi_\infty,
\end{equation}
\begin{equation}
    \label{u-feedback}
    u_*=-K_\infty x_*-\ell_\infty,
\end{equation}
and the  static balance equation
\[
    (A-BK_\infty)x+b-B\ell_\infty=0
\]
admits the unique solution $x=x_*$.
Moreover,
\begin{equation}
    \label{static-feedback}
    \pi_\infty^{\lambda,*}(\cdot\mid x_*) = \cNcal(u_*,\Lambda_*^\lambda).
\end{equation}
\end{proposition}

\begin{proof}
We first prove that \rf{zeta-feedback} holds.
Let $\eta^*:=\zeta^*-P_\infty x_*-\varphi_\infty$.
The $u$-equation in the KKT system \eqref{KKT-system} gives
\begin{equation}
\label{u-KKT}
\begin{aligned}
0
&=Sx_*+Ru_*+r+D^\top P_\infty(Cx_*+Du_*+\sigma)+B^\top \zeta^*\\
%&&\text{(KKT $u$-equation)}\\
%&=Sx_*+Ru_*+r+D^\top P_\infty(Cx_*+Du_*+\sigma)
  %+B^\top(P_\infty x_*+\varphi_\infty+\eta^*)\\
%&&\text{(definition of $\eta^*$)}\\
%&=(B^\top P_\infty+D^\top P_\infty C+S)x_*
  %+(R+D^\top P_\infty D)u_*
  %+(B^\top\varphi_\infty+D^\top P_\infty\sigma+r)+B^\top\eta^*\\
%&&\text{(rearrangement)}\\
&=N_\infty x_*+M_\infty u_*+d_\infty+B^\top\eta^*.
%&&\text{(definitions of $M_\infty,N_\infty,d_\infty$)}.
\end{aligned}
\end{equation}
Since $M_\infty>0$, \eqref{u-KKT} is equivalent to
\begin{equation}
    \label{u-error}
    u_*+M_\infty^{-1}(N_\infty x_*+d_\infty) = -M_\infty^{-1}B^\top\eta^*.
\end{equation}
Moreover, the stationary linear equation \eqref{affine-equation} can be written as
\begin{equation}
\label{affine-expansion}
\begin{aligned}
0
&=(A-BK_\infty)^\top\varphi_\infty
  +(C-DK_\infty)^\top P_\infty\sigma-K_\infty^\top r+P_\infty b+q\\
%&&\text{(by \eqref{affine-equation})}\\
%&=A^\top\varphi_\infty+P_\infty b+C^\top P_\infty\sigma+q
  %-K_\infty^\top(B^\top\varphi_\infty+D^\top P_\infty\sigma+r)\\
%&&\text{(by expansion)}\\
&=A^\top\varphi_\infty+P_\infty b+C^\top P_\infty\sigma+q
  -N_\infty^\top M_\infty^{-1}d_\infty.
%&&\text{($K_\infty,d_\infty$ definitions; $M_\infty^\top=M_\infty$)}.
\end{aligned}
\end{equation}
Using the $x$-equation in \eqref{KKT-system}, the algebraic Riccati equation \eqref{ARE}, the static drift balance, and \eqref{u-error}--\eqref{affine-expansion}, we obtain
\begin{equation*}
\begin{aligned}
0
&=Qx_*+S^\top u_*+q+C^\top P_\infty(Cx_*+Du_*+\sigma)+A^\top \zeta^*\\
%&&\text{(KKT-$x$)}\\
%&=(Q+A^\top P_\infty+C^\top P_\infty C)x_*
  %+(S^\top+C^\top P_\infty D)u_*+q+C^\top P_\infty\sigma
  %+A^\top\varphi_\infty+A^\top\eta^*\\
%&&\text{($\eta^*$ definition)}\\
%&=(-P_\infty A+N_\infty^\top M_\infty^{-1}N_\infty)x_*
  %+(N_\infty^\top-P_\infty B)u_*+q+C^\top P_\infty\sigma
  %+A^\top\varphi_\infty+A^\top\eta^*\\
%&&\text{(ARE; $N_\infty$ definition)}\\
&=-P_\infty(Ax_*+Bu_*)
  +N_\infty^\top(M_\infty^{-1}N_\infty x_*+u_*)
  +q+C^\top P_\infty\sigma+A^\top\varphi_\infty+A^\top\eta^*\\
%&&\text{(rearrangement)}\\
%&=P_\infty b+N_\infty^\top(M_\infty^{-1}N_\infty x_*+u_*)
  %+q+C^\top P_\infty\sigma+A^\top\varphi_\infty+A^\top\eta^*\\
%&&\text{(static balance)}\\
&=N_\infty^\top\bigl[u_*+M_\infty^{-1}(N_\infty x_*+d_\infty)\bigr]+A^\top\eta^*\\
%&&\text{(by \eqref{affine-expansion})}\\
&=-N_\infty^\top M_\infty^{-1}B^\top\eta^*+A^\top\eta^*\\
%&&\text{(by \eqref{u-error})}\\
%&=(A-BM_\infty^{-1}N_\infty)^\top\eta^*\\
%&&\text{($M_\infty^\top=M_\infty$)}\\
&=(A-BK_\infty)^\top\eta^*.
%&&\text{(by $K_\infty=M_\infty^{-1}N_\infty$)}.
\end{aligned}
\end{equation*}
%Since $P_\infty$ is the stabilizing solution of the algebraic Riccati equation, closed-loop stability implies that %$A-BK_\infty$ is invertible.
Since $P_\infty$ is the stabilizing solution of \eqref{ARE},
	Corollary~\ref{cor:Hurwitz} implies that $A-BK_\infty$ is invertible.
Hence $\eta^*=0$, and \eqref{zeta-feedback} follows.

Next, substituting $\eta^*=0$ into \eqref{u-KKT}, we get
\[
    \begin{aligned}
    u_* &= -M_\infty^{-1}(N_\infty x_*+d_\infty)
    %&&\text{(by \eqref{u-KKT})}\\
    = -K_\infty x_*-\ell_\infty.
    %&&\text{(by the feedback definitions)}.
    \end{aligned}
\]
This is \eqref{u-feedback}. Moreover,
\[
    \begin{aligned}
    (A-BK_\infty)x_*+b-B\ell_\infty &= Ax_*+B(-K_\infty x_*-\ell_\infty)+b \\
    %&&\text{(by expansion)}\\
    &= Ax_*+Bu_*+b
    %&&\text{(by \eqref{u-feedback})}\\
    =0.
    %&&\text{(by the static drift balance)}.
    \end{aligned}
\]
Since $A-BK_\infty$ is invertible, this linear equation has a unique solution.

Finally, by the infinite-horizon Gaussian feedback formula \eqref{Gaussian-feedback}, the mean compatibility \eqref{u-feedback}, and the covariance compatibility \eqref{optimal-covariance},
\[
    \begin{aligned}
    \pi_\infty^{\lambda,*}(\cdot\mid x_*) &= \cNcal(-K_\infty x_*-\ell_\infty,\Lambda_\infty^\lambda)
    %&&\text{(by \eqref{u-feedback})}\\
    = \cNcal(u_*,\Lambda_*^\lambda),
    % &&\text{(by mean and covariance compatibility)}.
    \end{aligned}
\]
which is \eqref{static-feedback}.
\end{proof}

The closed-loop system of \rf{state} induced by the finite-horizon optimal feedback is given by
\begin{equation}
    \label{finite-closedloop}\left\{
    \begin{aligned}
    &dX^{T,\lambda}(t) = \big[(A-BK_T(t))X^{T,\lambda}(t)+b-B\ell_T(t)\big]dt        \\
    &\qq\qq\q+ \big[(C-DK_T(t))X^{T,\lambda}(t)+\sigma-D\ell_T(t)\big]dW(t) \\
    &\qq\qq\q+ D\big(\Lambda_T^\lambda(t)\big)^{1/2}d\widetilde W(t), \qq t\in[0,T],\\
    & X^{T,\lambda}(0)=x.
    \end{aligned}\right.
\end{equation}
%The infinite-horizon Gaussian feedback distribution \eqref{Gaussian-feedback} is given by
%\[
%    \pi_\infty^{\lambda,*}(\cdot\mid x) = \cNcal\big(-K_\infty x-\ell_\infty,\Lambda_\infty^\lambda\big).
%\]
%The infinite-horizon closed-loop process induced by this feedback distribution and starting from the initial value $x$ is %defined by
The infinite-horizon closed-loop system induced by the feedback
	distribution \eqref{Gaussian-feedback}  is defined by
\begin{equation}
    \label{infinite-closedloop}
    \left\{
    \begin{aligned}
    &dX^{\infty,\lambda,x}(t) = \big[(A-BK_\infty)X^{\infty,\lambda,x}(t)+b-B\ell_\infty\big]dt\\
    &\qq\qq\qq+ \big[(C-DK_\infty)X^{\infty,\lambda,x}(t)+\sigma-D\ell_\infty\big]dW(t)\\
    &\qq\qq\qq+ D\big(\Lambda_\infty^\lambda\big)^{1/2}d\widetilde W(t), \qq t\in[0,\infty),\\
    & X^{\infty,\lambda,x}(0)=x.
    \end{aligned}\right.
\end{equation}

%By Proposition~\ref{prop:static-feedback-compatibility},
%the unique optimizer $(x_*,u_*)$  of the static problem \eqref{static-problem} satisfies
%$u_*=-K_\infty x_*-\ell_\infty$.
%Combining this with the static drift balance $Ax_*+Bu_*+b=0$, we get
By \eqref{u-feedback} and the first equation in
	\eqref{KKT-system}, we obtain
\begin{equation}
    \label{static-balance}
    (A-BK_\infty)x_*+b-B\ell_\infty=0.
\end{equation}
Again using the above feedback distribution, define the stochastic reference process by
\begin{equation}
    \label{reference-state}
    \left\{
    \begin{aligned}
    &dX^{*,\lambda}(t) = \big[(A-BK_\infty)X^{*,\lambda}(t)+b-B\ell_\infty\big]dt\\
    &\qq\qq\q+ \big[(C-DK_\infty)X^{*,\lambda}(t)+\sigma-D\ell_\infty\big]dW(t)\\
    &\qq\qq\q+ D\big(\Lambda_\infty^\lambda\big)^{1/2}d\widetilde W(t), \qquad t\in[0,\infty),\\
    &X^{*,\lambda}(0)=x_*.
    \end{aligned}\right.
\end{equation}

\begin{proposition}
Let \ref{ass:L2-stabilizable}--\ref{ass:JC} hold.
Let $(x_*,u_*)$ be the unique optimizer of the static problem \eqref{static-problem},
$X^{*,\lambda}$ be the unique solution of \rf{reference-state},
%and $\bar u^{*,\lambda}$ be the mean of the corresponding optimal control, that is
%$\bar u^{*,\lambda} := \bar u_\infty(X^{*,\lambda}) = -K_\infty X^{*,\lambda}-\ell_\infty$.
and let $\bar u^{*,\lambda}(t)$ be the mean of the feedback
	distribution \eqref{Gaussian-feedback} evaluated along
	$X^{*,\lambda}(t)$, that is,
	$\bar u^{*,\lambda}(t):=\bar u_\infty(X^{*,\lambda}(t))
	=-K_\infty X^{*,\lambda}(t)-\ell_\infty$.
Then
$$
\dbE X^{*,\lambda}(t)=x_*, \qq \dbE\bar u^{*,\lambda}(t) = u_*,\qq t\in[0,\infty).
$$
\end{proposition}

\begin{proof}
Taking expectation in \eqref{reference-state} and using \eqref{static-balance}, we obtain
\[
    \frac{d}{dt}\dbE X^{*,\lambda}(t) = (A-BK_\infty)\dbE X^{*,\lambda}(t)+b-B\ell_\infty = (A-BK_\infty)\big(\dbE X^{*,\lambda}(t)-x_*\big),
    \qquad \dbE X^{*,\lambda}(0)=x_*.
\]
Note that the above is a linear ordinary differential equation.
Then its unique solution is given by
\[
    \dbE X^{*,\lambda}(t)=x_*.
\]
Moreover, from \eqref{u-feedback}, we get
\[
    \dbE\bar u^{*,\lambda}(t) = -K_\infty\dbE X^{*,\lambda}(t)-\ell_\infty = -K_\infty x_*-\ell_\infty = u_*.
\]
This completes the proof.
\end{proof}

\begin{remark}
The means of $X^{*,\lambda}$ and $\bar u^{*,\lambda}$ stay at the static optimal points $x_*$ and $u_*$, respectively, although the processes themselves are generally not constant.
\end{remark}

By Corollary~\ref{cor:JC-Damm} and Lemma~\ref{lem:stability-equivalence}, the homogeneous closed-loop system
$dX(t)=(A-BK_\infty)X(t)dt+(C-DK_\infty)X(t)dW(t)$ is mean-square exponentially stable.
That is, there exist constants $\cC,\kappa>0$ such that
\begin{equation}
    \label{exponential-stability}
    \dbE|X(t)|^2\le \cC e^{-\kappa t}|X(0)|^2,\qquad t\ge0.
\end{equation}
Equivalently, one can choose $H\in\dbS^n_{++}$ such that
\begin{equation}
    \label{Lyapunov-equation}
    (A-BK_\infty)^\top H +H(A-BK_\infty) +(C-DK_\infty)^\top H(C-DK_\infty) \le -I_n.
\end{equation}
Let $h_-:=\lambda_{\min}(H)>0$ and $h_+:=\lambda_{\max}(H)>0$.
Then
\begin{equation}
    \label{H-norm}
    h_-|z|^2\le z^\top Hz\le h_+|z|^2, \qquad z\in\dbR^n.
\end{equation}

\begin{lemma}
\label{lem:moment-bound}
Let \ref{ass:L2-stabilizable}--\ref{ass:JC} hold.
Then there exists a constant $\cC>0$, independent of $t$, such that for any initial value $x\in\dbR^n$,
\begin{equation}
    \label{infinite-moment}
    \sup_{t\ge0}\dbE|X^{\infty,\lambda,x}(t)|^2 \le \cC(1+|x|^2+\lambda).
\end{equation}
In particular,
\begin{equation}
    \label{reference-moment}
    \sup_{t\ge0}\dbE|X^{*,\lambda}(t)|^2 \le \cC(1+\lambda).
\end{equation}
\end{lemma}

\begin{proof}
Let $Y(t)=X^{\infty,\lambda,x}(t)-x_*$.
Set
$\xi_*:=(C-DK_\infty)x_*+\sigma-D\ell_\infty=Cx_*+Du_*+\sigma$.
Then by \eqref{static-balance},  we have
\begin{equation*}\left\{
    \begin{aligned}
   & dY(t) =(A-BK_\infty)Y(t)dt+\big[(C-DK_\infty)Y(t)+\xi_*\big]dW(t)\\
    &\qq\qq+D(\Lambda_\infty^\lambda)^{1/2}d\widetilde W(t), \qquad t\in[0,\infty),\\
    &Y(0)=x-x_*.
    \end{aligned}\right.
\end{equation*}
Applying It\^o's formula to $Y(t)^\top HY(t)$, and using the independence of $W$ and $\widetilde W$, we obtain
\begin{align*}
    \frac{d}{dt}\dbE(Y(t)^\top HY(t)) ={}& \dbE\big[Y(t)^\top\big((A-BK_\infty)^\top H +H(A-BK_\infty)\big)Y(t)\big]\\
    &+\dbE\big[Y(t)^\top(C-DK_\infty)^\top H(C-DK_\infty)Y(t)\big]\\
    &+ 2\dbE\big[Y(t)^\top(C-DK_\infty)^\top H\xi_*\big] +\xi_*^\top H\xi_* +\tr\big(D^\top H D\Lambda_\infty^\lambda\big).
\end{align*}
%From \eqref{Lyapunov-equation}, we have
%\[
 %   \begin{aligned}
%    Y(t)^\top\big[(A-BK_\infty)^\top H+H(A-BK_\infty) +(C-DK_\infty)^\top H(C-DK_\infty)\big]Y(t)
 %  \le -|Y(t)|^2.
 %   \end{aligned}
%\]
%On the other hand, Young's inequality yields
By \eqref{Lyapunov-equation}, the quadratic terms are bounded by
	$-\dbE|Y(t)|^2$. For the cross term, Young's inequality gives
\[
    \begin{aligned}
    %2Y(t)^\top(C-DK_\infty)^\top H\xi_* &\le \frac12|Y(t)|^2 +2\big|(C-DK_\infty)^\top H\xi_*\big|^2\\
    2Y(t)^\top(C-DK_\infty)^\top H\xi_* &\le \frac12|Y(t)|^2+\cC|\xi_*|^2.
    \end{aligned}
\]
The exploration-noise trace term satisfies
\[
    \begin{aligned}
    \tr\bigl(D^\top H D\Lambda_\infty^\lambda\bigr) &= \frac{\lambda}{2} \tr\bigl(D^\top H D M_\infty^{-1}\bigr)
    \le \cC\lambda.
    \end{aligned}
\]
Thus
\[
    \frac{d}{dt}\dbE(Y(t)^\top HY(t)) \le -\frac12\dbE|Y(t)|^2+\cC(1+\lambda).
\]
From \eqref{H-norm}, we have $|y|^2\ge h_+^{-1}y^\top Hy$. Consequently,
\[
    \frac{d}{dt}\dbE(Y(t)^\top HY(t)) \le -c\dbE(Y(t)^\top HY(t))+\cC(1+\lambda).
\]
Gronwall's inequality then yields
\[
    \dbE|Y(t)|^2 \le \cC e^{-ct}|x-x_*|^2+\cC(1+\lambda), \qquad t\ge0.
\]
Since $x_*$ is fixed, \eqref{infinite-moment} follows. Setting $x=x_*$ gives \eqref{reference-moment}.
\end{proof}

\begin{theorem}
\label{thm:state-turnpike}
Let \ref{ass:L2-stabilizable}--\ref{ass:JC} hold.
Let $X^{T,\lambda}$ be the unique solution of the finite-horizon closed-loop system \eqref{finite-closedloop},
and let $X^{*,\lambda}$ be the stochastic reference process given by \eqref{reference-state}.
Then, for any fixed initial state $x$ and exploration weight $\lambda>0$, there exist constants $K=K_{x,\lambda}>0$ and $\mu>0$, independent of $T,t$, such that for all $0\le t\le T$,
\begin{equation}
    \label{state-turnpike}
    \dbE|X^{T,\lambda}(t)-X^{*,\lambda}(t)|^2 \le K\big(e^{-\mu t}+e^{-\mu(T-t)}\big).
\end{equation}
\end{theorem}

\begin{proof}
We first compare $X^{T,\lambda}$ with the infinite-horizon intermediate process $X^{\infty,\lambda,x}$ with the same initial value.
Let $Z(t)=X^{T,\lambda}(t)-X^{\infty,\lambda,x}(t)$.
Subtracting \eqref{infinite-closedloop} from \eqref{finite-closedloop}, we obtain
\begin{equation}
    \label{error-SDE}\left\{
    \begin{aligned}
    &dZ(t) =\big[(A-BK_\infty)Z(t)-B\Delta R_t\big]dt\\
    &\qq\qq+\big[(C-DK_\infty)Z(t)-D\Delta R_t\big]dW(t)+D\Delta\Lambda_t\,d\widetilde W(t), \qquad t\in[0,T],\\
    &Z(0)=0,
    \end{aligned}\right.
\end{equation}
where
\begin{equation*}
    \begin{aligned}
    \Delta R_t= \big(K_T(t)-K_\infty\big)X^{T,\lambda}(t) +\big(\ell_T(t)-\ell_\infty\big),\qq
    \Delta\Lambda_t= \big(\Lambda_T^\lambda(t)\big)^{1/2} - \big(\Lambda_\infty^\lambda\big)^{1/2}.
    \end{aligned}
\end{equation*}
By Proposition~\ref{prop:parameter-convergence}, there exist constants $\cC,\gamma_0>0$ such that
\begin{equation}
    \label{error-scale}
    \|K_T(t)-K_\infty\| +|\ell_T(t)-\ell_\infty| \le \cC\delta_t, \qquad \|\Delta\Lambda_t\|\le \cC\sqrt{\lambda}\,\delta_t,
\end{equation}
where $\delta_t=e^{-\gamma_0(T-t)}$.
Since $X^{T,\lambda}(t)=Z(t)+X^{\infty,\lambda,x}(t)$, we have
\begin{equation*}
    \Delta R_t = \big(K_T(t)-K_\infty\big)Z(t) + \big(K_T(t)-K_\infty\big)X^{\infty,\lambda,x}(t) + \big(\ell_T(t)-\ell_\infty\big).
\end{equation*}
Then by Lemma~\ref{lem:moment-bound}, we have
\begin{equation}
    \begin{aligned}
    \label{drift-error}
    \dbE|\Delta R_t|^2 &\le \cC\delta_t^2 \left( \dbE|Z(t)|^2 +\dbE|X^{\infty,\lambda,x}(t)|^2+1 \right) \\
    %&&\text{(by the decomposition and parameter bounds)}\\
    &\le \cC\delta_t^2\dbE|Z(t)|^2 +\cC(1+|x|^2+\lambda)\delta_t^2.
    %&&\text{(by Lemma~\ref{lem:moment-bound})}.
    \end{aligned}
\end{equation}

We next apply It\^o's formula to $Z(t)^\top HZ(t)$.
From \eqref{error-SDE}, we have
\begin{align*}
    \frac{d}{dt}\dbE(Z(t)^\top HZ(t)) ={}& \dbE\big[Z(t)^\top\big((A-BK_\infty)^\top H +H(A-BK_\infty)\big)Z(t)\big]\\
    &+\dbE\big[Z(t)^\top(C-DK_\infty)^\top H(C-DK_\infty)Z(t)\big]\\
    &- 2\dbE[Z(t)^\top HB\Delta R_t] - 2\dbE[Z(t)^\top(C-DK_\infty)^\top H D\Delta R_t]        \\
    &+ \dbE[(\Delta R_t)^\top D^\top H D\Delta R_t] + \tr\big((\Delta\Lambda_t)^\top D^\top H D\Delta\Lambda_t\big).
\end{align*}
Using \eqref{Lyapunov-equation} and Young's inequality, we get
\[
    \frac{d}{dt}\dbE(Z(t)^\top HZ(t)) \le -\frac12\dbE|Z(t)|^2 +\cC\dbE|\Delta R_t|^2 +\cC\|\Delta\Lambda_t\|^2.
\]
Let $V_Z(t)=\dbE(Z(t)^\top HZ(t))$.
Substituting \eqref{error-scale} and \eqref{drift-error} into the above,
we obtain
\[
    \dot V_Z(t) \le \big[-c+\cC\delta_t^2\big]V_Z(t) + \cC(1+|x|^2+\lambda)\delta_t^2, \qquad V_Z(0)=0.
\]
By Gronwall's inequality,
\[
    V_Z(t) \le \cC(1+|x|^2+\lambda) \int_0^t \exp\left(\int_s^t[-c+\cC\delta_r^2]dr\right) \delta_s^2ds.
\]
Note that
\[
    \begin{aligned}
    \int_s^t\delta_r^2dr &= \int_s^t e^{-2\gamma_0(T-r)}dr
    = \frac{e^{-2\gamma_0(T-t)}-e^{-2\gamma_0(T-s)}}{2\gamma_0}\le \cC.
    \end{aligned}
\]
Hence
\[
    \begin{aligned}
    \exp\left(\int_s^t[-c+\cC\delta_r^2]dr\right) &= e^{-c(t-s)} \exp\left(\cC\int_s^t\delta_r^2dr\right)\le \cC e^{-c(t-s)}.
    \end{aligned}
\]
It follows that
\begin{align*}
    V_Z(t) \le \cC(1+|x|^2+\lambda) \int_0^t e^{-c(t-s)}e^{-2\gamma_0(T-s)}ds
    %&= \cC(1+|x|^2+\lambda)e^{-2\gamma_0(T-t)} \int_0^t e^{-(c+2\gamma_0)(t-s)}ds\\
    \le \cC(1+|x|^2+\lambda)e^{-2\gamma_0(T-t)}.
\end{align*}
Using \eqref{H-norm}, we get
\begin{equation}
    \label{same-initial}
    \dbE|X^{T,\lambda}(t)-X^{\infty,\lambda,x}(t)|^2 \le \cC(1+|x|^2+\lambda)e^{-2\gamma_0(T-t)}.
\end{equation}
Together with Lemma~\ref{lem:moment-bound}, this also gives the uniform second-moment estimate for the finite-horizon  state process $X^{T,\lambda}$:
\begin{equation}
    \label{state-moment}
    \sup_{0\le t\le T}\dbE|X^{T,\lambda}(t)|^2 \le \cC(1+|x|^2+\lambda).
\end{equation}

On the other hand, $X^{\infty,\lambda,x}$ and $X^{*,\lambda}$ satisfy the same infinite-horizon closed-loop equation and are driven by the same Brownian motions.
Their only difference is the initial value.
Let $Y(t)=X^{\infty,\lambda,x}(t)-X^{*,\lambda}(t)$.
Then
\begin{equation*}
    dY(t)=(A-BK_\infty)Y(t)dt+(C-DK_\infty)Y(t)dW(t), \qq t\in[0,\infty),
\end{equation*}
with $Y(0)=x-x_*$.
By \eqref{exponential-stability}, we obtain the  estimate
\begin{equation}
    \label{initial-estimate}
    \dbE|X^{\infty,\lambda,x}(t)-X^{*,\lambda}(t)|^2 \le \cC e^{-\k t}|x-x_*|^2.
\end{equation}
Finally,
\[
    X^{T,\lambda}(t)-X^{*,\lambda}(t) = \big(X^{T,\lambda}(t)-X^{\infty,\lambda,x}(t)\big) + \big(X^{\infty,\lambda,x}(t)-X^{*,\lambda}(t)\big).
\]
Using $|a+b|^2\le 2|a|^2+2|b|^2$, and combining \eqref{same-initial} with \eqref{initial-estimate}, we obtain \eqref{state-turnpike}.
\end{proof}

We now derive a Wasserstein comparison for the control distributions from the state-process estimate.
The finite-horizon optimal feedback distribution is
\begin{equation*}
    \pi_T^{\lambda,*}(du\mid t,x) = \cNcal\big(\bar u_T(t,x),\Lambda_T^\lambda(t)\big)(du), \qquad \bar u_T(t,x)=-K_T(t)x-\ell_T(t),
\end{equation*}
and the infinite-horizon Gaussian feedback distribution defined by the long-time parameters is
\begin{equation*}
    \pi_\infty^{\lambda,*}(du\mid x) = \cNcal\big(\bar u_\infty(x),\Lambda_\infty^\lambda\big)(du), \qquad \bar u_\infty(x)=-K_\infty x-\ell_\infty.
\end{equation*}

To compare the finite-horizon control distribution with its infinite-horizon reference distribution, we use the $\cW_2$ distance \cite{Villani2009}.

\begin{definition}
Let $\mu,\nu$ be two Borel probability measures on $\dbR^m$ with finite second moments, and let $\Pi(\mu,\nu)$ be the set of all joint Borel probability distributions $\pi$ on $\dbR^m\times\dbR^m$ with first marginal $\mu$ and second marginal $\nu$.
The $\cW_2$ distance between $\mu$ and $\nu$ is defined by

\begin{equation*}
	\cW_2(\mu,\nu)
	=
	\left(
	\inf_{\pi\in\Pi(\mu,\nu)}
	\int_{\dbR^m\times\dbR^m}|u-v|^2\,d\pi(u,v)
	\right)^{1/2},
\end{equation*}
where $|\cdot|$ is the Euclidean norm on $\dbR^m$.
\end{definition}

When the feedback distributions are Gaussian,
the following formula, which is  taken from Peyr\'e--Cuturi \cite{PeyreCuturi2019}, allows us to compare them through their means and covariance matrices.

\begin{lemma}
\label{lem:Gaussian-W2}
Let $\mu_i=\cNcal(m_i,\Sigma_i)$, $i=1,2$, where $m_i\in\dbR^m$ and $\Sigma_i\in\dbS^m_{++}$.
\begin{align*}
    &\cW_2^2(\mu_1,\mu_2) = |m_1-m_2|^2+\mathfrak B^2(\Sigma_1,\Sigma_2),\\
    &\mathfrak B^2(\Sigma_1,\Sigma_2) = \tr\left[ \Sigma_1+\Sigma_2 -2\left(\Sigma_2^{1/2}\Sigma_1\Sigma_2^{1/2}\right)^{1/2} \right],\\
    &\mathfrak B(\Sigma_1,\Sigma_2) \le \|\Sigma_1^{1/2}-\Sigma_2^{1/2}\|_F.
\end{align*}
\end{lemma}

%With the above, we have the following Lemma.
Combining Lemma~\ref{lem:Gaussian-W2} with Proposition~\ref{prop:parameter-convergence} yields the following estimate.

\begin{lemma}
\label{lem:policy-W2}
Let \ref{ass:L2-stabilizable}--\ref{ass:JC} hold.
Then for any $x\in\dbR^n$ and $0\le t\le T$,
\begin{equation}
    \label{policy-distance}
    \cW_2^2\left( \pi_T^{\lambda,*}(\cdot\mid t,x), \pi_\infty^{\lambda,*}(\cdot\mid x) \right) \le \cC(1+|x|^2+\lambda)e^{-2\gamma(T-t)}.
\end{equation}
\end{lemma}

\begin{proof}
%By Proposition~\ref{prop:parameter-convergence}, we have
%
%$$
%|\bar u_T(t,x)-\bar u_\infty(x)|\le \cC(1+|x|)e^{-\gamma(T-t)}.
%$$
By Lemma~\ref{lem:Gaussian-W2}, we get
\begin{align*}
    &\cW_2^2\left( \pi_T^{\lambda,*}(\cdot\mid t,x), \pi_\infty^{\lambda,*}(\cdot\mid x) \right)\le |\bar u_T(t,x)-\bar u_\infty(x)|^2 + \left\| \big(\Lambda_T^\lambda(t)\big)^{1/2} - \big(\Lambda_\infty^\lambda\big)^{1/2}
    \right\|_F^2.
\end{align*}
%Using the  estimate \rf{sqrt-convergence} in Lemma~\ref{lem:covariance-convergence}, we get \eqref{policy-distance}.
Combining this with Proposition~\ref{prop:parameter-convergence}
	and \eqref{sqrt-convergence} yields \eqref{policy-distance}.
\end{proof}

Along the state processes $X^{T,\lambda}$ and $X^{*,\lambda}$,
define the finite-horizon control mean and the reference control mean by
\begin{equation}\label{def-u-u}
    \begin{aligned}
    &\bar u^{T,\lambda}(t) = \bar u_T(t,X^{T,\lambda}(t)) = -K_T(t)X^{T,\lambda}(t)-\ell_T(t),\\
    &\bar u^{*,\lambda}(t) = \bar u_\infty(X^{*,\lambda}(t)) = -K_\infty X^{*,\lambda}(t)-\ell_\infty.
    \end{aligned}
\end{equation}

\begin{lemma}
Under the conditions of Theorem~\ref{thm:state-turnpike}, we have
\begin{equation}
    \label{mean-turnpike}
    \dbE|\bar u^{T,\lambda}(t)-\bar u^{*,\lambda}(t)|^2 \le \cC e^{-\gamma t}|x-x_*|^2 + \cC(1+|x|^2+\lambda)e^{-\gamma(T-t)}.
\end{equation}
\end{lemma}

\begin{proof}
By \rf{def-u-u}, we have
\begin{align*}
    \bar u^{T,\lambda}(t)-\bar u^{*,\lambda}(t) = &-\big(K_T(t)-K_\infty\big)X^{T,\lambda}(t) -K_\infty\big(X^{T,\lambda}(t)-X^{*,\lambda}(t)\big)
    -\big(\ell_T(t)-\ell_\infty\big).
\end{align*}
Then by Proposition~\ref{prop:parameter-convergence}, we get
\begin{align*}
    \dbE|\bar u^{T,\lambda}(t)-\bar u^{*,\lambda}(t)|^2&\le \cC\|K_T(t)-K_\infty\|^2\dbE|X^{T,\lambda}(t)|^2 +\cC\dbE|X^{T,\lambda}(t)-X^{*,\lambda}(t)|^2
    +\cC|\ell_T(t)-\ell_\infty|^2\\
    %&\quad\le \cC e^{-2\gamma(T-t)}\dbE|X^{T,\lambda}(t)|^2 +\cC\dbE|X^{T,\lambda}(t)-X^{*,\lambda}(t)|^2\\
    %&\qquad+\cC e^{-2\gamma(T-t)}\\
    &\le \cC e^{-\gamma t}|x-x_*|^2 +\cC(1+|x|^2+\lambda)e^{-\gamma(T-t)}.
\end{align*}
After decreasing $\gamma$ if necessary, this is
\eqref{mean-turnpike}.
\end{proof}

Now we are ready to present the main result of this paper.

\begin{theorem}\label{Thm:TP}
Under the conditions of Theorem~\ref{thm:state-turnpike}, for any fixed initial state $x$ and exploration weight $\lambda>0$, there exist constants $K=K_{x,\lambda}>0$ and $\mu>0$, independent of $T,t$, such that for all $0\le t\le T$,
\begin{equation}
    \label{policy-turnpike}
    \begin{aligned}
    &\dbE|X^{T,\lambda}(t)-X^{*,\lambda}(t)|^2 + \dbE \cW_2^2\left( \pi_T^{\lambda,*}(\cdot\mid t,X^{T,\lambda}(t)), \pi_\infty^{\lambda,*}(\cdot\mid X^{*,\lambda}(t)) \right)\\
    &\le K\big(e^{-\mu t}+e^{-\mu(T-t)}\big).
    \end{aligned}
\end{equation}
\end{theorem}

\begin{proof}
By  Lemma~\ref{lem:policy-W2}, we obtain
\begin{align*}
    &\cW_2^2\left( \pi_T^{\lambda,*}(\cdot\mid t,X^{T,\lambda}(t)), \pi_\infty^{\lambda,*}(\cdot\mid X^{*,\lambda}(t)) \right)\\
    &\le 2\cW_2^2\left( \pi_T^{\lambda,*}(\cdot\mid t,X^{T,\lambda}(t)), \pi_\infty^{\lambda,*}(\cdot\mid X^{T,\lambda}(t)) \right)+ 2\cW_2^2\left( \pi_\infty^{\lambda,*}(\cdot\mid X^{T,\lambda}(t)), \pi_\infty^{\lambda,*}(\cdot\mid X^{*,\lambda}(t)) \right)\\
    &\le \cC\big(1+|X^{T,\lambda}(t)|^2+\lambda\big)e^{-2\gamma(T-t)} +\cC|X^{T,\lambda}(t)-X^{*,\lambda}(t)|^2.
\end{align*}
The last inequality also uses the fact that the two infinite-horizon Gaussian distributions have the same covariance, so their Wasserstein distance equals the distance between their means. Then by \eqref{state-moment} and Theorem~\ref{thm:state-turnpike}, we obtain
\[
    \begin{aligned}
    &\dbE \cW_2^2\left( \pi_T^{\lambda,*}(\cdot\mid t,X^{T,\lambda}(t)), \pi_\infty^{\lambda,*}(\cdot\mid X^{*,\lambda}(t)) \right)\\
    &\le \cC\big(1+\dbE|X^{T,\lambda}(t)|^2+\lambda\big)e^{-2\gamma(T-t)} +\cC\dbE|X^{T,\lambda}(t)-X^{*,\lambda}(t)|^2\\
    &\le K\big(e^{-\mu t}+e^{-\mu(T-t)}\big),
    \end{aligned}
\]
for some $K,\mu>0$.
Combining this with the state-process estimate in Theorem~\ref{thm:state-turnpike} yields \eqref{policy-turnpike}.
\end{proof}

In particular, by Proposition~\ref{prop:static-feedback-compatibility},
$\pi_\infty^{\lambda,*}(\cdot\mid x_*)=\cNcal(u_*,\Lambda_*^\lambda)$.
Thus, when the reference state is at the static center, the above result is compatible with the static Gaussian optimal distribution constructed in Section~\ref{sec:static-distribution-optimization}.

\begin{corollary}\label{coroll:TP}
For any $\delta\in(0,1/2)$, there exist constants $K=K_{x,\lambda,\delta}>0$ and $\mu=\mu_\d>0$, independent of $T,t$, such that for any $t\in[\delta T,(1-\delta)T]$,
\begin{equation*}
    \begin{aligned}
    \dbE|X^{T,\lambda}(t)-X^{*,\lambda}(t)|^2 + \dbE \cW_2^2\left( \pi_T^{\lambda,*}(\cdot\mid t,X^{T,\lambda}(t)), \pi_\infty^{\lambda,*}(\cdot\mid X^{*,\lambda}(t)) \right)
    \le K e^{-\mu T}.
    \end{aligned}
\end{equation*}
\end{corollary}

\begin{proof}
If $t\in[\delta T,(1-\delta)T]$, then $t\ge\delta T$ and $T-t\ge\delta T$.
Using $e^{-\mu t}\le e^{-\mu\delta T}$ and $e^{-\mu(T-t)}\le e^{-\mu\delta T}$ in the state estimate \eqref{state-turnpike} and the Wasserstein estimate \eqref{policy-turnpike} gives the result.
\end{proof}

%\section*{Appendix}
%\subsection*{Proof of Theorem \ref{thm:Damm}}
\appendix
\section{Proof of Theorem~\ref{thm:Damm}}
	Write $\|z\|_\infty^2=\dbE\int_0^\infty|z(t)|^2\,dt$.
	
	\textbf{$(i)\Rightarrow(ii)$.}
	We first construct $P_1$.
	Set
	\begin{equation*}
		\alpha=1+\|Q-S^\top R^{-1}S\|,\qquad
		Q_1=S^\top R^{-1}S+\alpha I_n,\qquad Q_2=Q_1-Q.
	\end{equation*}
	Then
	\[
	Q=Q_1-Q_2,\qquad Q_2\ge I_n,\qquad
	Q_1-S^\top R^{-1}S=\alpha I_n>0.
	\]
%	For $x\in\dbR^n$, consider the controlled homogeneous system \eqref{state-homo},
%	\[
%	\left\{
%	\begin{aligned}
%		&dX(t)=\bigl(AX(t)+Bu(t)\bigr)\,dt
%		+\bigl(CX(t)+Du(t)\bigr)\,dW(t),
%		\qquad t\ge0,\\
%		&X(0)=x,
%	\end{aligned}
%	\right.
%	\]
%	with the auxiliary cost functional obtained by replacing $Q$ with $Q_1$:
	For $x\in\dbR^n$, consider \eqref{state-homo} with the auxiliary
		cost functional obtained by replacing $Q$ with $Q_1$:
	\begin{equation*}
		J_{1,\infty}^0(x;u)
		=\dbE\int_0^\infty
		\bigl[X(t)^\top Q_1X(t)+2u(t)^\top SX(t)+u(t)^\top Ru(t)\bigr]\,dt,
	\end{equation*}
	where $u\in\mathcal U_{\mathrm{ad}}(x)$.
	This auxiliary problem satisfies the strong standard condition.
	By the $L^2$-stabilizability assumption and Sun--Wang--Yong~\cite[Proposition~3.1]{SunWangYong2022},
	%the corresponding algebraic Riccati equation admits a positive definite stabilizing solution $P_1$ satisfying
		there exists a positive definite stabilizing solution
		$P_1$ of the auxiliary algebraic Riccati equation
	\begin{equation}\label{P1}
		A^\top P_1+P_1A+C^\top P_1C+Q_1=K_1^\top M_1K_1,
	\end{equation}
	where
	\begin{equation}\label{P1-gain}
		\begin{aligned}
			M_1&=R+D^\top P_1D>0,
			&\qquad K_1&=M_1^{-1}(B^\top P_1+D^\top P_1C+S),\\
			A_1&=A-BK_1,
			&\qquad C_1&=C-DK_1,
		\end{aligned}
	\end{equation}
	and $[A_1,C_1]$ is mean-square exponentially stable.
	
	Let $u(t)=v(t)-K_1X(t)$.
	Substitution into \eqref{state-homo} gives
	\begin{equation}\label{stable-state}
		\left\{
		\begin{aligned}
			&dX(t)=\bigl[A_1X(t)+Bv(t)\bigr]\,dt
			+\bigl[C_1X(t)+Dv(t)\bigr]\,dW(t),
			\qquad t\ge0,\\
			&X(0)=x.
		\end{aligned}
		\right.
	\end{equation}
	Since $[A_1,C_1]$ is $L^2$-stable, every $v\in L^2_{\dbF}(0,\infty;\dbR^m)$ yields a state $X\in L^2_{\dbF}(0,\infty;\dbR^n)$; see Sun--Yong~\cite{SunYong2018}.
	Consequently, $u=v-K_1X$ belongs to $\mathcal U_{\mathrm{ad}}(x)$.
	Conversely, every $u\in\mathcal U_{\mathrm{ad}}(x)$ and its corresponding state $X$ determine $v=u+K_1X\in L^2_{\dbF}(0,\infty;\dbR^m)$.
	These two transformations are inverses of each other.
	
	Applying It\^o's formula to $X(t)^\top P_1X(t)$ along the solution of \eqref{stable-state} and using \eqref{P1}--\eqref{P1-gain}, we obtain
	\begin{equation}\label{P1-Ito}
		\begin{aligned}
			&\frac{d}{dt}\dbE[X(t)^\top P_1X(t)]=\dbE\bigl[
			X(t)^\top(A^\top P_1+P_1A+C^\top P_1C)X(t)\\
			&\q
			+2u(t)^\top(B^\top P_1+D^\top P_1C)X(t)
			+u(t)^\top D^\top P_1Du(t)
			\bigr]\\
			&=\dbE\bigl[
			X(t)^\top(K_1^\top M_1K_1-Q_1)X(t)
			+2u(t)^\top(M_1K_1-S)X(t)
			+u(t)^\top(M_1-R)u(t)
			\bigr]\\
			% &={}\dbE\bigl[
			% (K_1X(t))^\top M_1(K_1X(t))
			% +2u(t)^\top M_1(K_1X(t))
			% +u(t)^\top M_1u(t)-X(t)^\top Q_1X(t)
			% -2u(t)^\top SX(t)-u(t)^\top Ru(t)
			% \bigr]\\
			&=\dbE\bigl[
			(u(t)+K_1X(t))^\top M_1(u(t)+K_1X(t))
			-X(t)^\top Q_1X(t)-2u(t)^\top SX(t)-u(t)^\top Ru(t)
			\bigr]\\
			&=\dbE\bigl[
			v(t)^\top M_1v(t)-X(t)^\top Q_1X(t)
			-2u(t)^\top SX(t)-u(t)^\top Ru(t)
			\bigr].
		\end{aligned}
	\end{equation}
	Substituting $Q=Q_1-Q_2$ and integrating over $[0,T]$, we obtain
	\begin{equation}\label{finite-cost}
		\begin{aligned}
			J_T^0(x;u)+\dbE[X(T)^\top P_1X(T)]
			={}&x^\top P_1x
			+\dbE\int_0^T
			\bigl[v(t)^\top M_1v(t)-X(t)^\top Q_2X(t)\bigr]\,dt.
		\end{aligned}
	\end{equation}
	Since $[A_1,C_1]$ is $L^2$-stable and $v$ is square-integrable, we have $\dbE|X(T)|^2\to0$ as $T\to\infty$.
	The square-integrability of $X$, $u$, and $v$ also ensures that the running-cost integrals converge.
	Thus, letting $T\to\infty$ in \eqref{finite-cost} yields
	\begin{equation}\label{split-cost}
		\begin{aligned}
			&J_\infty^0(x;u)=x^\top P_1x+J_{2,\infty}(x;v),\\
			&J_{2,\infty}(x;v):=\dbE\int_0^\infty
			\bigl[v(t)^\top M_1v(t)-X(t)^\top Q_2X(t)\bigr]\,dt.
		\end{aligned}
	\end{equation}
	
	Now we construct $P_2$.
	Fix $v\in L^2_{\dbF}(0,\infty;\dbR^m)$, and let $X$ solve \eqref{stable-state} with $X(0)=0$.
	By the preceding control transformation, $u(t)=v(t)-K_1X(t)$ belongs to $\mathcal U_{\mathrm{ad}}(0)$.
	The coercivity condition \eqref{C-infinite} and the cost decomposition \eqref{split-cost} therefore give
	\[
	J_{2,\infty}(0;v)=J_\infty^0(0;u)
	\ge\delta\|X\|_\infty^2.
	\]
	Consequently,
	\begin{align*}
		\lambda_{\min}(M_1)\|v\|_\infty^2
		&\le\dbE\int_0^\infty v(t)^\top M_1v(t)\,dt
		=J_{2,\infty}(0;v)
		+\dbE\int_0^\infty X(t)^\top Q_2X(t)\,dt\\
		&\le J_{2,\infty}(0;v)+\|Q_2\|\|X\|_\infty^2
		\le\left(1+\frac{\|Q_2\|}{\delta}\right)J_{2,\infty}(0;v).
	\end{align*}
	It follows that
	\begin{equation}\label{transformed-UC}
		J_{2,\infty}(0;v)\ge
		\frac{\delta\lambda_{\min}(M_1)}{\delta+\|Q_2\|}
		\|v\|_\infty^2,
		\qquad \forall v\in L^2_{\dbF}(0,\infty;\dbR^m).
	\end{equation}
	For the system \eqref{stable-state}, the uncontrolled system $[A_1,C_1]$ is $L^2$-stable, and \eqref{transformed-UC} establishes uniform convexity of the associated cost functional $J_{2,\infty}$.
	By Sun--Yong~\cite[Theorem~5.3]{SunYong2018}, its algebraic Riccati equation admits a stabilizing solution $P_2\in\dbS^n$ satisfying
	\begin{equation}\label{P2}
		A_1^\top P_2+P_2A_1+C_1^\top P_2C_1-Q_2
		=K_2^\top M_2K_2,
	\end{equation}
	where
	\begin{equation}\label{P2-gain}
		M_2=M_1+D^\top P_2D>0,\qquad
		K_2=M_2^{-1}(B^\top P_2+D^\top P_2C_1),
	\end{equation}
	and $[A_1-BK_2,C_1-DK_2]$ is $L^2$-stable.
	
	%Finally, we combine the two Riccati equations.
	%Set $P=P_1+P_2$ and $K=K_1+K_2$.
	%Using $A=A_1+BK_1$ and $C=C_1+DK_1$, we obtain
	We now combine \eqref{P1} and \eqref{P2}.
		Set $P=P_1+P_2$ and $K=K_1+K_2$.
		Using \eqref{P1-gain}, we obtain
	\begin{align*}
		&A^\top P_2+P_2A+C^\top P_2C
		=A_1^\top P_2+P_2A_1+C_1^\top P_2C_1\\
		&\q+K_1^\top B^\top P_2+P_2BK_1
		+K_1^\top D^\top P_2C_1+C_1^\top P_2DK_1
		+K_1^\top D^\top P_2DK_1\\
		&=Q_2+K_2^\top M_2K_2
		+K_1^\top(B^\top P_2+D^\top P_2C_1)
		+(P_2B+C_1^\top P_2D)K_1
		+K_1^\top D^\top P_2DK_1\\
		&=Q_2+K_2^\top M_2K_2+K_1^\top M_2K_2
		+K_2^\top M_2K_1+K_1^\top D^\top P_2DK_1.
	\end{align*}
	Adding \eqref{P1} % namely,
	% \[
	% A^\top P_1+P_1A+C^\top P_1C+Q_1=K_1^\top M_1K_1,
	% \]
	and using $Q=Q_1-Q_2$, we obtain
	\begin{equation}\label{riccati-sum}
		\begin{aligned}
			A^\top P+PA+C^\top PC+Q
			={}&K_1^\top(M_1+D^\top P_2D)K_1
			+K_1^\top M_2K_2+K_2^\top M_2K_1+K_2^\top M_2K_2\\
			={}&(K_1+K_2)^\top M_2(K_1+K_2)
			=K^\top M_2K.
		\end{aligned}
	\end{equation}
	Moreover,
	\begin{equation}\label{riccati-sum-gain}
		\begin{aligned}
			&R+D^\top PD
			=R+D^\top P_1D+D^\top P_2D=M_2>0,\\
			&B^\top P+D^\top PC+S
			=M_1K_1+B^\top P_2+D^\top P_2(C_1+DK_1)\\
			&\q=M_1K_1+M_2K_2+D^\top P_2DK_1
			=(M_1+D^\top P_2D)K_1+M_2K_2=M_2K,\\
			&K=(R+D^\top PD)^{-1}(B^\top P+D^\top PC+S).
		\end{aligned}
	\end{equation}
	Substituting the expressions for $M_2$ and $K$ in \eqref{riccati-sum-gain} into \eqref{riccati-sum} yields \eqref{ARE} for $P$.
	Note that
	\[
	A-BK=A_1-BK_2,\qquad C-DK=C_1-DK_2,
	\]
	so $[A-BK,C-DK]$ is $L^2$-stable and hence mean-square exponentially stable.
	Taking $P_\infty=P$ and $K_\infty=K$ proves $(ii)$.
	
	\textbf{$(ii)\Rightarrow(iii)$.}
	Fix $x\in\dbR^n$ and $u\in\mathcal U_{\mathrm{ad}}(x)$,
	and let $X$ solve \eqref{state-homo}.
	Set $v(t)=u(t)+K_\infty X(t)$.
	Then the state equation can be written as
	\[
	\left\{
	\begin{aligned}
		&dX(t)=\bigl[(A-BK_\infty)X(t)+Bv(t)\bigr]\,dt\\
		&\qq\qq+\bigl[(C-DK_\infty)X(t)+Dv(t)\bigr]\,dW(t),
		\qquad t\ge0,\\
		&X(0)=x.
	\end{aligned}
	\right.
	\]
	Since $u\in\mathcal U_{\mathrm{ad}}(x)$, both $X$ and $u$ are square-integrable on $[0,\infty)$, and hence $v\in L^2_{\dbF}(0,\infty;\dbR^m)$.
	The $L^2$-stability of $[A-BK_\infty,C-DK_\infty]$, together with the square-integrability of $v$, implies that $\lim_{T\to\infty}\dbE|X(T)|^2=0$.
	Applying It\^o's formula as in \eqref{P1-Ito} and using \eqref{ARE}--\eqref{gain}, we obtain
	\begin{align*}
		&\frac{d}{dt}\dbE[X(t)^\top P_\infty X(t)]
		+\dbE\bigl[
		X(t)^\top QX(t)+2u(t)^\top SX(t)+u(t)^\top Ru(t)
		\bigr]
		=\dbE[v(t)^\top M_\infty v(t)].
	\end{align*}
	Integrating over $[0,T]$ and letting $T\to\infty$ yields
	\begin{equation}\label{verification}
		J_\infty^0(x;u)
		=x^\top P_\infty x
		+\dbE\int_0^\infty v(t)^\top M_\infty v(t)\,dt.
	\end{equation}
	
 	Now take $x=0$. Applying Lemma~\ref{lem:state-estimate} to the stable closed-loop system and letting the time horizon tend to infinity, we obtain a constant $\cC>0$, independent of $u$, such that
	\[
	\|X\|_\infty^2\le\cC\|v\|_\infty^2.
	\]
	Since $u(t)=v(t)-K_\infty X(t)$, we have
	\[
	\|u\|_\infty^2
	\le 2\|v\|_\infty^2
	+2\|K_\infty\|^2\|X\|_\infty^2
	\le \bigl(2+2\cC\|K_\infty\|^2\bigr)\|v\|_\infty^2.
	\]
	Combining these estimates with \eqref{verification} and $M_\infty>0$, we obtain
	\[
	\begin{aligned}
	J_\infty^0(0;u)
	&\ge \lambda_{\min}(M_\infty)\|v\|_\infty^2
	\ge
	\frac{\lambda_{\min}(M_\infty)}
	{2+\cC(1+2\|K_\infty\|^2)}
	\bigl(\|X\|_\infty^2+\|u\|_\infty^2\bigr).
	\end{aligned}
	\]
	This proves $(ii)\Rightarrow(iii)$.
	
	The implication $(iii)\Rightarrow(i)$ follows immediately by dropping the nonnegative control term from the lower bound in $(iii)$.

	Under these equivalent conditions, \eqref{verification} and $M_\infty>0$ imply that
	\[
	J_\infty^0(x;u)\ge x^\top P_\infty x,
	\qquad \forall u\in\mathcal U_{\mathrm{ad}}(x),
	\]
	with equality if and only if $u=-K_\infty X$.
	The mean-square exponential stability of the closed-loop system ensures that this feedback control is admissible.
	Thus \eqref{optimal} gives the optimal value and the unique optimal control for each initial state $x$.
	
	To prove uniqueness of the stabilizing solution, let $\widehat P\in\dbS^n$ be another solution of \eqref{ARE} such that $R+D^\top\widehat P D>0$ and the associated closed-loop system is mean-square exponentially stable.
	Applying the same verification argument to $\widehat P$ yields
	\[
	x^\top\widehat P x=V_\infty^0(x)=x^\top P_\infty x,
	\qquad \forall x\in\dbR^n.
	\]
	Since both matrices are symmetric, this identity implies $\widehat P=P_\infty$.

\section*{Acknowledgments}
YH is supported in part by National Key Research and Development Program of China (2025YFA1016903), National Natural Science Foundation of China
(12571327, 12426311), Project of Educational Commission of Guangdong Province (2025ZDZX2059), Shenzhen Science and Technology Program (RCJC20221008092753082).
HW is supported in part by National Natural Science Foundation of China (12522121, 12631016).
JX is supported in part by  National Key Research and Development Program of China (2022YFA1006102) and
National Natural Science Foundation of China (12471418).

\ms

\noindent
\textbf{Statement on AI use:}
This manuscript was written entirely by the authors. AI-assisted tools  were used solely for language polishing and grammatical correction, and not for generating scientific content.

\begingroup
\let\ss\originalss

\endgroup

\end{document}